\documentclass[11pt]{article}
\usepackage{amsmath} 
\usepackage{lmodern}
\usepackage[T1]{fontenc}
\usepackage[english]{babel}
\usepackage[utf8]{inputenc}
\usepackage[a4paper]{geometry}
\usepackage{amssymb}
\usepackage[hypcap]{caption}
\usepackage{amsmath, amsthm, amsfonts, amscd}
\usepackage[nottoc,numbib]{tocbibind}
\usepackage{mathrsfs}
\usepackage{float}
\usepackage{graphicx}
\usepackage[labelformat=empty]{caption}
\usepackage{microtype}
\usepackage{todonotes}
\usepackage{makeidx}
\usepackage{verbatim}
\usepackage{hyperref}
\hypersetup{
colorlinks=false,
linkbordercolor=red,
pdfborderstyle={/S/U/W 1}
}
\usepackage{mdwlist}
\usepackage{mathabx}
\usepackage{mathtools}
\usepackage{titlesec}

\newtheorem{thm}{\scshape{Theorem}}[section]
\newtheorem{cor}[thm]{\scshape{Corollary}}
\newtheorem{prop}[thm]{\scshape{Proposition}}
\newtheorem{fact}[thm]{\scshape{Fact}}
\newtheorem{lemma}[thm]{\scshape{Lemma}}

\newtheorem{question}[thm]{\scshape{Question}}

\newtheorem*{theorem*}{Theorem}

\newtheorem*{LCLT}{Local Central Limit Theorem (2.1.3 \cite{Lawler})}
\newtheorem*{SZL}{Schwartz-Zippel Lemma}
\newtheorem*{prop*}{Proposition}
\newtheorem*{cor*}{Corollary}

\theoremstyle{definition}
\newtheorem{dfn}[thm]{\scshape{Definition}}
\newtheorem{con}[thm]{\scshape{Convention}}
\newtheorem{example}[thm]{\scshape{Example}}

\newtheorem{remark}[thm]{\scshape{Remark}}

\def\G{\Gamma}

\def\subset{\subseteq}

\def\={\cong}

\def\e{\epsilon}

\def\mb{\mathbf}
\def\mc{\mathcal}
\def\mbb{\mathbb}

\def\l{\lambda}

\def\ab{{\it ab}}
\def\Is{{\it Is}}
\def\ker{{\it ker}}
\def\G_0{G/\Is(\gamma_{3}(G))}

\def\End{\it End}
\newcommand{\ol}[1]{{#1}{}^{\mbb{Q}}}

{}

\title{Full rank presentations and nilpotent groups: structure, Diophantine problem, and genericity}
\author{Albert Garreta\footnote{University of the Basque Country,  Spain, \emph{garreta.a@gmail.com} (corresponding author).}, Alexei Miasnikov\footnote{Stevens Institute of Technology, NJ, USA, \emph{amiasnikov@gmail.com}}, and Denis Ovchinnikov\footnote{Stevens Institute of Technology, NJ, USA, \emph{dovchinn@stevens.edu}\newline This work was supported by the Mathematical Center in Akademgorodok.}
}

\begin{document}

\maketitle

\date{}

\begin{abstract}
We study finitely generated nilpotent groups $G$ given by full rank finite presentations $\langle A \mid R\rangle_{\mc{N}_c}$ in the  variety $\mathcal{N}_c$ of nilpotent groups of class at most $c$, where $c \geq 2$.  We prove that if the deficiency $|A| - |R| $ is at least $2$ then the group $G$ is virtually free nilpotent, it is quasi finitely axiomatizable (in particular, first-order rigid), and it is almost (up to finite factors) directly indecomposable. One of the main results of the paper is that the Diophantine problem in nilpotent groups given by full rank finite presentations $\langle A \mid R\rangle_{\mc{N}_c}$ is undecidable if $|A| - |R|  \geq 2$  and decidable otherwise.  We show that this class of groups is rather large since finite presentations asymptotically almost surely have full rank, so a random nilpotent group in the few relators model has a full rank presentation asymptotically almost surely.  Full rank presentations give one a useful  tool to approach random nilpotent groups and study their properties. Note, that the results above significantly improve our understanding of the Diophantine problem in finitely generated nilpotent groups: from a few special examples of groups with undecidable Diophantine problem we got to the place where we know that the Diophantine problem in all "typical" nilpotent groups is also undecidable.  
\end{abstract}

\tableofcontents

\section{Introduction}

In this paper we study finitely generated nilpotent groups $G$ given by full rank finite presentations $\langle a_1, \ldots, a_n \mid r_1, \ldots, r_m\rangle_{\mc{N}_c}$ in the  variety $\mathcal{N}_c$ of nilpotent groups of class at most $c$, where $c \geq 2$. We prove that if $m \leq n-2$ then the group $G$ has undecidable Diophantine problem, it is quasi finitely axiomatizable (in particular, first-order rigid), it is almost directly indecomposable, and virtually free nilpotent. The class of such groups is rather large. Indeed, it turns out that for fixed $n$ and $m$ a finite presentation $\langle a_1, \ldots, a_n \mid r_1, \ldots, r_m\rangle_{\mc{N}_c}$ has full rank asymptotically almost surely. In particular, random nilpotent groups (in the few relators model) have full rank presentations asymptotically almost surely. Hence, they asymptotically almost surely satisfy all the properties mentioned above. 

\subsection{Groups with full rank presentations}

Let $A = \{a_1, \ldots,a_n\}$ be a finite alphabet,  $A^{-1} = \{a_1^{-1}, \ldots,a_n^{-1}\}$, $A^{\pm 1} = A \cup A^{-1}$, $(A^{\pm 1})^\ast$  the set of all (finite) words in $A^{\pm 1}$, and  $R = \{r_1, \ldots,r_m \}$  a finite subset of $(A^\pm)^\ast$. We fix this notation for the rest of the paper. 

A pair $(A,R)$ is called  a {\em finite  presentation}, we  denote it by $\langle A \mid R\rangle$ or $\langle a_1, \ldots,a_n \mid r_1, \ldots,r_m \rangle$. If $\mathcal V$ is a variety or a quasivariety of groups then a finite presentation $\langle A \mid R\rangle$ determines a group $G = F_{\mathcal V}(A)/\langle \langle R \rangle \rangle$, where  $F_{\mathcal V}(A)$ is a free group in $\mathcal V$ with basis $A$ and $\langle \langle R \rangle \rangle$ is the normal subgroup of $F_{\mathcal V}(A)$ generated by $R$. In this case we write $G = \langle A \mid R\rangle_{\mathcal V}$.  There are several objects related to the presentation $\langle A \mid R\rangle_{\mc{V}}$.  The {\em relation matrix} $M(A,R)$ of the presentation $\langle A \mid R\rangle_{\mc{V}}$
is an  $m \times n$ matrix whose $(i,j)$-th entry is the sum of the exponents of the $a_j$'s  that occur in $r_i$. It was introduced by Magnus in \cite{Magnus} (see also \cite{LS}, Chapter II.3,   for its ties  to relation modules in groups). For  a finite presentation $\langle A \mid R\rangle_{\mc{V}}$   the number $|A| - |R|$, if non-negative, is called the {\em deficiency} of the presentation (see \cite{LS}, Chapter II.2 for a short survey on groups with positive deficiency).  Sometimes we denote the relation matrix $M(A,R)$ simply by $M$.  
The matrix $M(A,R)$ has {\em full rank} if its rank is equal to $\min\{|A|,|R|\}$, i.e., it is the maximum possible. 

Groups given in  the variety ${\mathcal N}_c$  of all nilpotent groups of class $\leq c$ by full rank presentations have a rather restricted structure, as witnessed by the following result from Section 4.1.

\medskip \noindent
{\bf Theorem 4.2.}
{\it Let $G$ be a finitely generated nilpotent group of class $c \geq 1$ given by a finite full rank presentation $G = \langle A \mid R\rangle_{{\mathcal N}_c} = \langle a_1, \ldots,a_n \mid  r_1, \ldots,r_m \rangle_{{\mathcal N}_c}$.
 Then  there exists a subset $A_0\subseteq A$ with $|A_0| = n-m$ ($A_0=\emptyset$ if $m\geq n$) such that the following holds:
\begin{enumerate}
\item If $m\geq n$, then $G$ is finite.
\item If $m=n-1$, then $\langle A_0\rangle$ is infinite cyclic and has finite index in $G$.%
\item If $m \leq n-2$, then $\langle A_0 \rangle$ is a free $c$-step nilpotent subgroup of rank $n-m$ which has finite index in $G$. 
\end{enumerate}
 Furthermore, $A_0$ can be chosen to be  any subset $\{a_{i_1}, \ldots,a_{i_{n-m}}\}$ of $A$ such that the rank of the matrix $M(A,R)$ coincides with the rank of the matrix obtained from $M(A,R)$ after removing its $i_1, \dots, i_{n-m}$-th columns.
}

\medskip \noindent
The result above complements the Generalized Freiheissatz for ${\mathcal N}_c.$  In \cite{Romanovskii} Romanovskii  proved that if a   finitely generated nilpotent group given in the variety  ${\mathcal N}_c, c \geq 1,$ by a finite presentation $\langle A \mid R\rangle_{\mc{N}_c}$ of deficiency $d  \geq 1$ then there is a subset of generators $A_0 \subseteq A$ with $|A_0| = d$ which freely generates a free nilpotent subgroup $H = \langle A_0\rangle$  of nilpotency class $c$.   The items 2) and 3)  above show that if in addition the presentation $\langle A \mid R\rangle_{\mc{N}_c}$ has full rank then the subgroup $H$ has finite index in $G$. Note that in this case the subgroup $H$ can be found algorithmically. 

Note also that in the assumptions of  Theorem 4.2, if in addition $G/G'$ has trivial torsion subgroup (all entries in the Smith normal form of the relation matrix are 1's),  then $G$ is trivial if $m\geq n$; and free nilpotent of rank $n-m$ if $m\leq n-1$.

Some other structural results  on finitely generated nilpotent groups $G$ given by finite presentations of full rank in the variety $\mathcal N_2$ are given in Section 4.3. In particular, we show that if  the deficiency of the presentation is at least 2 then the  center $Z(G)$ of $G$ coincides with the derived subgroup $G'$ of $G$.     

In general, if in a variety $\mathcal N_c$, $c\geq 2$, a nilpotent group $G$ has a full rank finite presentation with deficiency $\geq 2$ then $Z(G) \leq G'$. This leads to some further interesting properties of these groups. Following \cite{Lub} we say that a finitely generated group $G$ is {\em first-order rigid} if for any finitely generated group $H$ elementary equivalence $G \equiv H$ implies isomorphism $G \simeq H$.  Furthermore, a finitely generated group $G$ is called {\em quasi finitely axiomatizable } or {\em QFA}, if there is a sentence $\phi$ of group theory that holds in $G$ and such that  for for any finitely generated group $H$ if $\phi$ holds in $H$ then $G \simeq H$.  We show in Section 4.4 (Theorem \ref{th:FAQ}) that every nilpotent group $G$ given in a  class $\mathcal{N}_c$, $c \geq 2$, by a finite full rank presentation of deficiency at least $2$ is QFA, hence first-order rigid.

In another direction, we show in Theorem \ref{t: direct_decompositions} that in any direct decomposition of a nilpotent group $G$ given in a  class $\mathcal{N}_c$, $c \geq 2$, by a finite full rank presentation of deficiency $\geq 1$,  all, but one, direct factors are finite.  So such groups are directly indecomposable, if ignoring finite factors.

\subsection{Diophantine problems}

In Section 3 we study the Diophantine problem in finitely generated nilpotent groups. This  is a continuation of research started in \cite{GMO2}.

  Recall, that the \emph{Diophantine problem} in an algebraic structure $\mathcal{A}$ (denoted $\mc{D}(\mc{A})$)  is the task to determine whether or not  a given finite system of equations with constants in $\mc{A}$ has a  solution in $\mathcal{A}$. $D(\mathcal{A})$ is \emph{decidable} if there is an algorithm that given a finite system $S$ of equations with constants in $\mathcal{A}$ decides whether or not $S$ has a solution in $\mc{A}$. Furthermore, $\mc{D}(\mc{A})$ is  \emph{reducible} to  $\mc{D}(\mc{M})$, for   another structure $\mc{M}$,  if there is an algorithm that for any finite system of equations $S$  in $\mc{A}$ computes a finite system of equations $S_{\mc{M}}$ in $\mc{M}$ such that $S$ has a solution in $\mc{A}$ if and only if $S_{\mc{M}}$  has a solution in $\mc{M}$. 

 Note that due to the classical result of Davis,  Putnam, Robinson and Matiyasevich, the Diophantine problem  $\mc{D}(\mathbb{Z})$ in the ring of integers $\mathbb{Z}$ is undecidable \cite{mat, DPR}.  Hence if $\mc{D}(\mathbb{Z})$ is reducible to $\mc{D}(\mc{M})$, then  $\mc{D}(\mc{M})$ is also  undecidable.  

To prove  that  $\mc{D}(\mc{A})$ reduces to  $\mc{D}(\mc{M})$ for some structures  $\mc{A}$ and  $\mc{M}$  it suffices to show that  $\mc{A}$ is  interpretable by equations (or \emph{e-interpretable}) in $\mc{M}$.  E-interpretability is a variation of the classical notion of the first-order interpretability, where instead of arbitrary first-order formulas finite systems of equations are used as the interpreting formulas (see Definition \ref{d: e-int} for details).  The main relevant property of such interpretations is that if $\mc{A}$ is e-interpretable in $\mc{M}$ then $\mc{D}(\mc{A})$  is reducible to $\mc{D}(\mc{M})$ by a polynomial time many-one reduction (Karp reductions).

For finitely generated nilpotent groups $G$ we use the following  technique for e-interpretability of $\mathbb{Z}$ in $G$, which resembles  some earlier arguments by  Duchin, Liang, and Shapiro in \cite{Duchin}, and by  Romankov in \cite{Romankov100}. An element $g$ of a group $G$ is said to be   \emph{centralizer-small} (\emph{c-small}) if  its centralizer has the following form $C_G(g)=\{ g^t z \ | \ t\in \mbb{Z}, \ z \in Z(G)\}$ and $C_G(g)/Z(G)$ is infinite cyclic.  
We show in Theorem \ref{mainThmEq2} that if $G$ is  a finitely generated $2$-step nilpotent group, which has   a centralizer-small element then the ring $\mbb{Z}$  is e-interpertable in $G$.    In fact, in this case the largest ring of scalars of $G$ is isomorphic to $\mbb{Z}$ (Theorem \ref{CSmallMaxZ}). 
Furthermore, we prove that if $G$ is a  nilpotent group given in the variety ${\mathcal N}_2$, by a finite full rank presentation of deficiency $\geq 2$ then $G$  has  centralizer-small elements.   

The following technique extends the method of c-small elements to finitely generated nilpotent groups $G$ of arbitrary nilpotency class $c \geq 2$. Namely, it turns out  that the 2-nilpotent factor-group $G/\gamma_3(G)$ of $G$ modulo the third term $\gamma_3(G)$ of the lower central series of $G$ is e-interpretable in $G$  (see Proposition \ref{any_step}). Moreover, if $G$ is given by a finite full rank presentation in the variety ${\mathcal N}_c$ then  $G/\gamma_3(G)$ also has a full rank presentation in the variety ${\mathcal N}_2$.  Now, the argument above  allows one to describe completely the  Diophantine problems in nilpotent non-abelian groups given by full rank presentations.  Note that the Diophantine problem in finitely generated abelian groups is decidable (linear algebra).

\medskip \noindent
{\bf Theorem 4.14.}
{\it Let $G$ be a  nilpotent group given in the variety ${\mathcal N}_c$, $c\geq 2$,  by a finite full rank presentation of deficiency $d$.  Then the following holds: 
\begin{enumerate}
\item If $d \geq 2$, then the ring $\mathbb{Z}$ is e-interpretable in $G$ and the Diophantine problem in $G$ is undecidable.
\item If $d\leq 1$ then the Diophantine problem in $G$ is decidable. 
\end{enumerate}
}

\medskip
The theorem  above complements  some earlier results from \cite{GMO1}. Indeed, we showed in \cite{GMO1} that if $G$ is a finitely generated nilpotent  group, which is not virtually abelian, then some finitely generated ring ${\mathcal O}_G$ of algebraic integers is e-interpretable in $G$.  There  is a well-known  conjecture in number theory (see, for example, \cite{DL,PZ}) that states that the Diophantine problem in rings of algebraic integers is undecidable.    Hence we conjectured in \cite{GMO1} that  Diophantine problem is  undecidable in  finitely generated nilpotent  groups, which are  not virtually abelian.  Theorem 4.3 shows that if $G$ has a full rank presentation then ${\mathcal O}_G \simeq \mathbb{Z}$, so the conjecture above holds in such $G$.

\subsection{Random nilpotent groups}

In \cite{Duchin}, Cordes, Duchin, Duong, Ho, and Sanchez outlined a general approach to random finitely generated nilpotent groups, which is a natural analog of the classical \emph{few-relators} and the \emph{density}  models of random finitely presented groups. 

Recall, that  to get a random finitely presented group  on $m$ generators one may take a finite alphabet $A = \{a_1, \dots, a_n\}$  and  add  a finite set of "random" relators $R \subseteq (A^{\pm 1})^*$, thus obtaining a "random" finitely presented group given by the presentation $\langle A \mid R\rangle$.  More precisely, every relator in $R$ is chosen among all words of a certain length $\ell$ in the alphabet $A^{\pm 1}$ with uniform probability. The length $\ell$ is thought of as an integer variable that tends to infinity, and the  number  of chosen relators is taken to be a function of $\ell$. For instance, $|R| = (2m+1)^{d\ell}$ ($0<d<1$) in the \emph{density model}, whereas $|R|$ is constant in the \emph{few-relators model}. Then one can consider  the probability $p_{\ell}$ that a group $G= \langle A \mid R\rangle$   satisfies some property $P$, for a fixed length $\ell$. The limit $p=\lim_{\ell \to \infty} p_{\ell}$, if it exists, is called the \emph{asymptotic probability} that $G$ satisfies $P$. If $p=1$, then  $G$ is said to satisfy $P$ \emph{asymptotically almost surely} (a.a.s.)  For example, a well-known result of Gromov \cite{Gromov} states that, in the density model,  $G$ is hyperbolic  if $d<1/2$, and  finite  if $d > 1/2$, a.a.s. See \cite{Olivier} for more information on random finitely presented groups.

Observe that the approach above can be utilized in any class of groups where finite presentations make sense, in particular, in any variety of groups.  Following \cite{Duchin}, for every $c \in \mathbb{N}$ one can take the variety ${\mathcal N}_c$ of all nilpotent groups of nilpotency class at most $c$ and consider random finitely presented groups in ${\mathcal N}_c$ with respect to the density of  the few-relators models.    In \cite{Duchin} it is proved (among other results) that in the class ${\mathcal N}_c$ the group $G = \langle A \mid R\rangle_{\mc{N}_c}$  is trivial a.a.s.\ (as $\ell \to \infty$) if and only if $|R|$ tends to infinity as a function of $\ell$. In particular, the density model in the class ${\mathcal N}_c$ yields trivial groups a.a.s. for any density parameter $d$. Because of this  in the class ${\mathcal N}_c$ the density model does not look very interesting. 

A  completely different model of random nilpotent groups was introduced by  Delp, Dymarz and Schaffer-Cohen in \cite{Dymarz}, where the authors obtain random 2-generated torsion-free nilpotent groups  as  subgroups of  groups of unitriangular matrices generated by two random words in the standard generators. This model is based on  the classical result  that any finitely generated     torsion-free nilpotent group embeds into a  suitable unitriangular group $UT_n(\mathbb{Z})$.  

Yet another model of randomness in the class ${\mathcal N}_c$  was considered by the authors in \cite{GMO2}.  The idea of this approach is based on finite descriptions of finitely generated torsion-free  nilpotent groups via their polycyclic presentations, or, equivalently, via Malcev's bases and their structural constants.

In this paper, following \cite{Duchin},  we study the structure of 
random nilpotent groups in  ${\mathcal N}_c$   in the few-relators model. The main result here (Theorems \ref{B} and \ref{B'}) is that    a random finite presentation (in the class of all groups, or any other variety  of groups) in the few-relators model a.a.s. has full rank. Hence,   the results above on finitely generated groups given in the class ${\mathcal N}_c$ by finite full rank presentations apply to the random nilpotent group asymptotically almost surely.  We mention here only one such result  and refer to Section \ref{sec:random} for precise formulations of the rest.

\medskip \noindent
{\bf Theorem 6.5.}
{\it Let $n, m, c \in \mbb{N}$, and let $G$ be a finitely generated $c$-step nilpotent group given by a presentation $\langle a_1, \dots, a_n \mid r_1, \dots, r_m \rangle_{\mc{N}_c}$, where $c\geq 2$, and all relators $r_i$ have length $\ell$. Then  the following holds asymptotically almost surely  as $\ell \to \infty$: 
\begin{enumerate}
\item If $m \leq n-2$, then the ring $\mathbb{Z}$ is e-interpretable in $G$  and the Diophantine problem in $G$ is undecidable.
\item If $m \geq n-1$  then the Diophantine problem in $G$ is decidable. 
\end{enumerate}
}

\medskip
Theorem 6.4 makes a rather big step in our understanding of decidability of equations in finitely generated nilpotent groups: from a few sporadic examples where the Diophantine problem was known to be undecidable \cite{GMO1}  - to the case when  asymptotically  almost surely all of them have the  Diophantine problem undecidable.

\section{Preliminaries}

\subsection{Nilpotent groups}

Let $G$ be a group. We use the following notation:   $[g,h] = g^{-1}h^{-1}gh$ is  the commutator of $g,h \in G$ and $g^h = h^{-1}gh$ is the conjugate of $g$ by $h$.  For a subset $X \subseteq G$ by $\langle X \rangle$ and $\langle \langle X\rangle \rangle$ we denote, respectively,  the subgroup and the normal subgroup  of $G$ generated by $X$;  for subgroups $H, K \leq G$  by $[H,K]$ we denote the subgroup in $G$ generated by all commutators $[h,k]$, where $h \in H, k \in K$. In particular,  $[G,G]$ is the derived subgroup of $G$, also denoted by $G'$.     More generally, we define the terms of the lower central series of $G$ inductively by 
$\gamma_1(G)=G$,  $\gamma_{\ell+1}(G) = [ \gamma_{\ell}(G), G]$, in particular, $\gamma_2(G)= G'$. Recall that a  group $G$ is $c$-step nilpotent (or nilpotent of class $c$) if $\gamma_{c+1}(G) = 1$, but $\gamma_{c}(G) \neq 
1$. An element from $\gamma_k(G)$ is called a \emph{$k$-fold} commutator. If $G$ is generated by a set $A$ then a \emph{basic} $k$-fold commutator (with respect to $A$) is an element of the form $[g_1, \dots, g_k]$ for $g_i \in A^{\pm 1}$ for all $i$.  The center of $G$ is denoted by $Z(G)$.   

The identities 
$$
[ab, c]=[a,c]^b[b,c], \ \ \ [a,bc]=[a,c][a,b]^c$$
 hold in any group $G$. This together with the inclusion $\gamma_c(G) \leq Z(G)$ can be used to prove by induction on $c$ that if $G$ is $c$-step nilpotent then
\begin{align}
&[g_1, \dots, g_{i-1}, g_i g_i', g_{i+1}, \dots, g_c] \label{e: multilinear}\\ = [g_1,& \dots, g_{i-1}, g_i, g_{i+1}, \dots, g_c][g_1, \dots, g_{i-1}, g_i', g_{i+1}, \dots, g_c]\nonumber
\end{align}
for all $1 \leq i \leq s$ and any elements $g_1, \dots, g_c$, $g_i'$. That is, informally, the $c$-fold commutator of $G$ behaves `multilinearly'.
We will use this fact implicitly from now on.

Let  $F = F(A)$ be a free group with basis  $A =\{a_1, \dots, a_n\}$.   Then the group $N_c(A) = F/\gamma_{c+1}(F)$ is a free $c$-step nilpotent group with basis $A$. It is, indeed, a free group with basis $A$ in the variety $\mathcal{N}_c$ of all nilpotent groups of class at most $c$. The number $n$ is the \emph{rank} of  $N_c(A)$. Sometimes we write $N_{c,n}$ for a free $c$-step nilpotent group of rank $n$, or just $N$.  Every $n$-generated $c$-step nilpotent group $G$ is isomorphic to the factor-group $N_{c,n}/\langle \langle R\rangle \rangle$ for some finite subset $R \subseteq N_{c,n}$. In this case we write 
$$
G = \langle A \mid R\rangle_{\mathcal{N}_c}
$$
and say that $G$ is given by a finite  presentation $\langle A \mid R\rangle_{\mc{N}_c}$ in the variety $\mathcal{N}_c$.  
 
 For $c=2$, the set $\mathcal{A}=A \cup \{[a_i, a_j] \mid i< j \}$ is a so-called \emph{Malcev basis} of $N = N_{2,n}(A)$. Any $g\in N$ admits the expression 
\begin{equation}\label{malcev_rep}
g= 
a_1^{\alpha_1} \dots a_m^{\alpha_n}\prod_{1\leq i < j \leq n} [a_i, a_j]^{\gamma_{i,j}}, \quad \left(\alpha_k, \gamma_{i,j} \in \mathbb{Z}\right),
\end{equation}
which is \emph{unique} up to the order of the commutators $[a_i, a_j]$. This follows from the fact that $\{[a_i, a_j] \mid i< j\}$ forms a basis of $N'$, while $A$ projects onto a basis of $N/ N'$. Usually, \eqref{malcev_rep} is called the \emph{Malcev representation} of $g$ (with respect to $\mathcal{A}$), and the integers $\alpha_k$, $\gamma_{i,j}$ are called the \emph{Malcev coordinates} of $g$. It is possible to extend this definition to any finitely generated free nilpotent group (of any nilpotency class), or even to any finitely generated torsion-free  nilpotent group \cite{KaMi}. However, the definition just given will suffice for our purposes. 

The following results  are well known (see, for example, \cite{Segal, Warfield}).

\begin{lemma}\label{SegalCor2} 
Let $G$ be a finitely generated nilpotent group. Then the following holds:
\begin{enumerate}
\item [1)] If  $G'$ is finite then the center of $G$ has finite index in $G$. In particular, $G$ is virtually abelian. \item [2)] If $G/G'$ is finite then $G$ is finite.
\end{enumerate}
\end{lemma}	

\subsection{$\mbb{Q}$-completions of nilpotent groups}\label{s: Q_completions}

In order to prove Theorem \ref{A} we will need some basic facts about $\mathbb{Q}$-completions of finitely generated nilpotent groups. We follow Section 8 from \cite{Warfield} (alternatively, see \cite{Kargapolov}).

A group $G$ is a \emph{$\mbb{Q}$-group} if any equation of the type $x^n = g$, where $n\in \mbb{N}\smallsetminus\{0\}$ and $g \in G$, has a unique solution in $G$. The language of \emph{$\mbb{Q}$-groups}, denoted $\mc{G}_{\mbb{Q}}$, consists of the language of groups together with countably many ``$n$-th root'' operations $\sqrt[n]{}$ ($n\in \mbb{N}\smallsetminus\{0\}$), so that a $\mbb{Q}$-group $G$ may be viewed as a structure in the language $\mc{G}_{\mbb{Q}}$, where   $(\sqrt[n]{g})^n=g$ for all $g$ in $G$  and all $n\in \mbb{N}\smallsetminus\{0\}$. We say that $G$ is a finitely $\mbb{Q}$-generated $\mbb{Q}$-group if $G$ is finitely generated in the language $\mc{G}_{\mbb{Q}}$.    We will use the notation  $\langle \cdot \rangle_{\mbb{Q}}$ to denote generation as a $\mbb{Q}$-group.

\begin{dfn}
Let $G$ be a nilpotent group. A $\mbb{Q}$-group $\ol{G}$ is a \emph{$\mbb{Q}$-completion} of $G$ if there is a homomorphism $\phi: G\to \ol{G}$ such that:
\begin{enumerate}
    \item $\phi(G)$ $\mbb{Q}$-generates $\ol{G}$.
    \item For any homomorphism $f: G\to H$ into a $\mbb{Q}$-group $H$ there is a unique homomorphism $\ol{f}:\ol{G}\to H$ such that $\phi \circ\ol{f} = f$.
\end{enumerate}
\end{dfn}

Observe that $\ol{G}$ is necessarily nilpotent of at most the same class as $G$. Nilpotent $\mbb{Q}$-groups form a variety in the language $\mc{G}_{\mbb{Q}}$ and hence one can speak of free nilpotent $\mbb{Q}$-groups. We denote such variety by $\mc{N}_c^{\mbb{Q}}$

The next result summarizes all basic properties of $\mbb{Q}$-completions that we will need in order to prove Theorem \ref{A}.

\begin{fact}\label{p: properties_Q_completions}
Let $G$ and $H$ be nilpotent groups. Then the following hold:
\begin{enumerate}
    \item  Every nilpotent group $G$ has a unique $\mbb{Q}$-completion, and if the group is  torsion-free then $\phi: G\to \ol{G}$ is an embedding. In the latter case we identify $G$ with $\phi(G)$.
    \item The functor $G\to G^{\mbb{Q}}$ respects composition of homomorphisms and short exact sequences. Hence if $K\leq G$ then $\ol{K}\leq\ol{G}$, and if  $H\cong G/K$ then $\ol{H}\cong \ol{G}/\ol{K}$.
    \item  In particular, if $G$ admits a presentation $\langle a_1, \dots, a_n\mid r_1, \dots, r_m\rangle_{\mc{N}_c}$ then $\ol{G}$ is presented by $\langle a_1, \dots, a_n\mid r_1, \dots, r_m\rangle_{\mc{N}_c^{\mbb{Q}}}$.
    \item If $H\leq G$ then $\ol{H}=\ol{G}$ if and only if $H$ is of finite index in $G$. 
    \item Let $G$ be a nilpotent group generated by $a_1,\dots,a_n$. Then $G$ is a free nilpotent group of class $c$ freely generated by $a_1, \dots, a_n$  if and only if the $\mbb{Q}$-completion $\ol{G}$  is free in the  variety of nilpotent  $\mbb{Q}$-groups of class $c$ freely generated by $\phi(a_1), \dots, \phi(a_n)$. 
\end{enumerate}
\end{fact}
\begin{proof}
Property 1 is proved in  \cite[Theorems 8.5 and 8.11]{Warfield}. Property 2 is a consequence of  \cite[Theorems 8.5 and 8.9]{Warfield}, and Property 3 is a simple consequence of Property 2.  Property 4 follows from \cite[Theorem 8.2]{Warfield}. To prove Property 5 note that, by Property 1, if $G$ is a free nilpotent group of class $c$ on $n$ generators  then $\ol{G}$ satisfies the universal property in the class of  nilpotent $\mbb{Q}$-groups of class $c$ on $n$ generators, hence in this case $\ol{G}$ is a  free nilpotent $\mbb{Q}$-group. In the opposite direction, if $G$ is not free, then $G=N/K$, where $N$ is a free nilpotent group freely generated by $a_1, \dots, a_n$ and $K$ a non-trivial subgroup of $N$. Since $K$ is torsion-free it embeds into its $\mbb{Q}$-completion $K^\mbb{Q}$, so $K^\mbb{Q}$ is a non-trivial $\mbb{Q}$-group.  Hence $\ol{G}\cong\ol{N}/\ol{K}$ is a proper factor of $\ol{N}$, and thus $\ol{G}$ cannot be a free nilpotent $\mbb{Q}$-group freely generated by $\phi(a_1), \dots, \phi(a_n)$.  
\end{proof}

\subsection{E-interpretability}

Let $\mathcal{M}=\left(M; f_i, r_j, c_k \mid i,j,k \right)$ be an algebraic structure, where $M$ is the universe set of $\mathcal{M}$, and  $f_i, r_j, c_k$ are the function, relation, and constant symbols of $\mathcal{M}$.
 In this paper all algebraic structures are either groups or rings, thus in the definition below $\mc{M}$ may be thought of as one of these. In what follows we often  use non-cursive boldface letters to denote tuples of elements: e.g.\ $\mathbf{a}=(a_1, \dots, a_n)$.
\begin{dfn}\label{d: e-def} 
 A set $D \subset M^m$ is called \emph{definable by equations} in $\mathcal{M}$, or \emph{e-definable} in $\mathcal{M}$, if there exists a finite system of equations in the language of $\mc{M}$
$
\Sigma_D(x_1,\ldots,x_m, y_1, \dots, y_k)$ on variables $(x_1, \dots, x_m, y_1, \dots, y_k)=(\mathbf{x}, \mathbf{y})$ and some constants from $\mc{M}$,  and such that 
for any tuple $\mathbf{a}\in M^m$, one has that  $\mathbf{a} \in D$ if and only if the system $\Sigma_D(\mathbf{a}, \mb{y})$ on variables $\mathbf {y}$ has a solution in $\mathcal{M}$. In this case $\Sigma_D$ is said to \emph{e-define} $D$ in $\mc{M}$.
\end{dfn}
From the viewpoint of number theory, an e-definable set is a Diophantine set, since it is defined in $\mc{M}$ by the formula $\exists y_1 \ldots \exists y_n\Sigma_D(\mathbf{x}, \mb{y})$. From the perspective of algebraic geometry, an e-definable set is a projection  of an affine algebraic set. 
\begin{dfn}\label{d: e-int}
An algebraic structure $\mathcal{A}= \left(A; f, \dots, r, \dots, c, \dots \right)$  is called \emph{e-interpretable} in an algebraic structure $\mathcal{M}$ if there exists $n\in \mathbb{N}$, a subset $D \subseteq \mathcal{M}^n$ and an onto map  (called the \emph{interpreting} map)
$
\phi: D \twoheadrightarrow \mathcal{A},
$ 
such that:
\begin{enumerate}
\item $D$ is e-definable in $\mathcal{M}$. 
\item  For every function $f=f(x_1, \dots, x_n)$ in the language of $\mathcal{A}$, the preimage by $\phi$ of the graph of $f$, i.e.\ the set $\{(x_1, \dots, x_k, x_{k+1}) \mid \phi(x_{k+1}) = f(x_1, \dots, x_k)\}$, is e-definable in $\mathcal{M}$. 
\item For every relation $r$ in the language of $\mathcal{A}$, and also for  the equality relation $=$ on $\mathcal{A}$,  the preimage by $\phi$ of the graph of $r$ is e-definable in $\mathcal{M}$.
\end{enumerate}
\end{dfn}

\begin{example}\label{group_interpretations}
%
The center $Z(G)$ of a finitely generated group $G =\langle g_1, \dots, g_n \rangle$ is e-definable in $G$ as a set. Indeed,  $x\in G$ belongs to $Z(G)$ if and only if it commutes with all $g_i$'s, and hence $Z(G)$ (seen as a set) is defined in $G$ by means of the following system of equations on the single variable $x$:
$$
\bigwedge_{i=1}^{n}  \left( xg_i = g_ix \right).
$$
If, additionally, we regard $Z(G)=\left(Z(G); \cdot, {}^{-1}, 1 \right)$ as an algebraic structure with  operations and constants inherited from $G$, then $Z(G)$ is still e-interpretable in $G$ with interpreting map ${\it id}: Z(G) \subseteq G \to Z(G)$, $id(g)=g$. 
\end{example}

\begin{con}
In general, any subgroup $H$ of a group $G$ that is e-definable in $G$ as a set is also e-interpretable in $G$ as a group through the identity map. 

In this situation when we say that $H$ is e-definable in $G$ we tacitly understand that $H$ is also e-interpretable in $G$ as a group.
\end{con}
\begin{example}
Let $G$ be a group with  finite $[x,y]$-width $n$ (see below in this section). Then any $g\in G'$ can be written as a product of exactly $n$ commutators (adding trivial ones if necessary), and thus $G'$ is e-definable in $G$ by means of the equation $$x=\left[x_1,y_1\right]\cdots\left[x_n,y_n\right]\left[x_{n+1},y_{n+1}\right]^{-1}\cdots\left[x_{2n},y_{2n}\right]^{-1}.$$  

Further examples can be found in Proposition \ref{extra_prop2}, and in the \emph{Verbal width} subsection below. 
\end{example}

The following is a fundamental property of e-interpretability. Intuitively it states that if $\mc{A}$ is e-interpretable in $\mc{M}$, then any system of equations in $\mc{A}$ can be `encoded' as a system of equations in $\mc{M}$.
\begin{lemma}
\label{RedLemma}
Let  $\mathcal{A}$ be e-interpretable in $\mathcal{M}$ with  the interpreting map  $\phi: D \to \mathcal{A}$ (in the notation of the Definition \ref{d: e-int}). Then for every system of equations $S(\mathbf{x})$ in $\mathcal{A}$, there exists a system of equations $S^*(\mathbf{y}, \mathbf{z})$ in $\mathcal{M}$, such that if $\mathbf{b}, \mathbf{c}$ is a solution to $S^*(\mathbf{y}, \mathbf{z})$ in $\mathcal{M}$ then $\phi(\mathbf{b})$ is a solution to $S(\mathbf{x})$ in $\mathcal{A}$. Moreover, any solution $\mathbf{a}$ to $S(\mathbf{x})$ in $\mathcal{A}$ arises in this way, i.e.\  $\mathbf{a} = \phi(\mathbf{b})$ for some solution $(\mathbf{b}, \mathbf{c})$ to $S^*(\mathbf{y}, \mathbf{z})$ in $\mathcal{M}$. Furthermore, there is a polynomial time algorithm that constructs the system $S^*(\mathbf{y}, \mathbf{z})$ when given a system $S(\mathbf{x})$. 
\end{lemma}
\begin{proof}
It suffices to follow step by step the proof of Theorem 5.3.2 from \cite{Hodges}, which states that the above holds when $\mc{A}$ is interpretable by first order formulas in $\mc{M}$. 
\end{proof}
 Now we state two key consequences of Lemma \ref{RedLemma}. 
\begin{cor}\label{RedCor}
If $\mathcal{A}$ is e-interpretable in $\mathcal{M}$, then $\mc{D}(\mc{A})$ is reducible to $\mc{D}(\mc{M})$. Consequently, if $\mc{D}(\mathcal{A})$  is undecidable, then $\mc{D}(\mathcal{M})$ is undecidable as well.
\end{cor}
\begin{cor}\label{transitivity}
E-interpetability is a transitive relation. I.e.\ if $\mc{A}_1$ is e-intepretable in $\mc{A}_2$, and $\mc{A}_2$ is e-interpretable in $\mc{A}_3$, then $\mc{A}_1$ is e-interpretable in $\mc{A}_3$.
\end{cor}

The next result will prove valuable later on.
\begin{prop}\label{extra_prop1}
Let $H$ be a normal subgroup of a group $G$. If  $H$ is e-definable in $G$ then the natural map  $\pi: G \to G/H$ is an e-interpretation of $G/H$ in $G$.   Consequently, $\mc{D}(G/H)$ is reducible to $\mc{D}(G)$. 
\end{prop}
\begin{proof}
Solving an  equation $w(\mathbf{z}) = 1$ on $n$ variables $\mathbf{z}$  in  $G/H$   is equivalent to  finding $n$ elements $\mathbf{g} \in G^n$ such that $w(\mathbf{g}) \in H$. Since $H$ is e-definable in $G$, there exists a system of equations $S_w=S_w(\mathbf{x}_w, \mathbf{y}_w)$ in $G$ such that $w(\mathbf{g}) \in H$ if and only if $S_w(\mathbf{g}, \mathbf{y}_w)$  has a solution $\mathbf{y}_w$ in $G$ (of course, here $\mb{x}_w$ is a tuple of $n$ variables, each one taking values in $G$). 

Now let $w_{\odot}$ and $w_{=}$ be the equations $xy H = zH$ and $xH=yH$ in $G/H$, respectively. Then the preimage of the graph of the multiplication operation of $G/H$ is e-definable in $G$ by means of  the system of equations $S_{w_{\odot}}$. Similarly, the preimage of the graph of the equality relation in $G/H$ is e-definable in $G$ by means of $S_{w_{=}}$. In conclusion, $G/H$ is e-interpretable in $G$. The last statement of the proposition follows directly from Corollary \ref{RedCor}. 
\end{proof}

We next use Proposition \ref{extra_prop1} to study  the torsion subgroup of a nilpotent group, "equation-wise".

\begin{prop}\label{extra_prop2}
Let $T$ be the torsion subgroup of a finitely generated nilpotent group $G$. Then $T$ is e-definable in $G$, and, consequently, $\mc{D}(G/T)$ is reducible to $\mc{D}(G)$.
\end{prop}
\begin{proof}
The torsion subgroup of such $G$ has finite order, say $n$ (\cite{Segal}, Chapter 1, Corollary 10). Hence, the equation $x^n=1$   e-defines $T$ in $G$. The result now follows from Proposition \ref{extra_prop1}.
\end{proof}

\subsection{Verbal width}\label{verbal_width} 
Let $w=w(x_1, \dots, x_m)$ be a word in  an alphabet  $\{x_1, \dots, x_m\}^{\pm 1}$. The \emph{$w$-verbal subgroup} $w(G)$ of a group $G$ is the subgroup  $ \langle w(g_1,\ldots,g_m)\mid g_i\in G \rangle$.  $G$ is said to have \emph{finite $w$-width} if there exists an integer $n$, such that every  $g \in w(G)$ can be expressed as a product of at most $n$ elements of the form $w(g_1,\ldots,g_m)^{\pm 1}$, i.e., if, for all $g\in w(G)$,
$$
g= \prod_{i=1}^{n'} w(g_{1}^i, \dots, g_m^i)^{\epsilon_i} \ \hbox{for some} \ g_j^i \in G, \epsilon_i\in \{-1, +1\}, \ \hbox{and} \ n'\leq n.
$$  
In this case, each $g\in w(G)$ can be expressed as a product of \emph{exactly} $n$ elements of the form $w(g_1, \dots, g_m)^{\pm 1}$ (adding trivial elements $w(1, \dots, 1)$ if necessary). Hence, $w(G)$ is e-definable  in $G$ by  the following equation
$$
x=\prod_{i=1}^n (w(y_{i1}, \dots, y_{im}) w(z_{i1}, \dots, z_{im})^{-1})
$$
on variables $x$ and $\{y_{ij}, z_{ij} \mid 1\leq i\leq n, \ 1\leq j \leq m\}$.

For example, if $G$ is a finitely generated nilpotent group, then $\gamma_i(G)$ (the $i$-th  subgroup of the lower central series of $G$) is $w_i$-verbal for $w_i = [x_1, \dots, x_{i}]$. In this case $G$ has finite $w_i$-verbal width for any  $i \in \mathbb{N} \smallsetminus \{0\}$ \cite{Romankov10} \cite{Segal10}. In particular, $\gamma_i(G)$ is e-definable in $G$.

\begin{prop}\label{any_step}
Let $G$ be a finitely generated nilpotent group with lower central series $G=\gamma_1(G) \unrhd \gamma_2(G) \unrhd \ldots \unrhd \gamma_c(G) \unrhd \gamma_{c+1}(G)=1$. Denote by $H_{ij}$ the quotient group $\gamma_i(G)/\gamma_j(G)$,  $1\leq i\leq j\leq s+1$. Then $H_{ij}$ is e-interpretable in $G$, and consequently  $\mc{D}(H_{ij})$ is reducible to $\mc{D}(G)$, for all $i\leq j$.
\end{prop}
\begin{proof}
Since $\gamma_k(G)$ is e-definable and normal in $G$ for all $k$, we have that $H_{ij}$ is e-interpretable in $G$, and  the Diophantine problem
$\mc{D}(H_{ij})$ is reducible to $\mc{D}(G)$ for all $i\leq j$, by Proposition \ref{extra_prop1}.
\end{proof}
This proposition allows one to reduce $\mc{D}(H_{1,3})$ to $\mc{D}(G)$. This is of particular interest because $H_{1,3}$ is a finitely generated  2-step nilpotent group, possibly non-abelian.  There are other $H_{ij}$ with these characteristics, for example $H_{i,3i}$, for any $i$.

\section{Diophantine problem in nilpotent groups}\label{section: Diophantine_problem}

In this section we introduce the notion of a largest ring of scalars of a $2$-step nilpotent group $G$. This  ring is canonically associated to $G$ and it is  e-interpretable in $G$. We prove that if $G$ contains a  so-called centralizer-small element  then such ring is isomorphic to $\mbb{Z}$, hence obtaining undecidability of the Diophantine problem in $G$.

A ring of scalars  of a finitely generated $2$-step nilpotent group $G$ is a commutative, associative  ring $R$ with identity such that $R$ acts faithfully by endomorphisms  on $G/Z(G)$ and on $G'$, and such that the commutator map between abelian groups%
\begin{align}
[\cdot,\cdot]: G/Z(G) \times G/Z(G) &\rightarrow G' \nonumber\\ 
\left(gZ(G), hZ(G)\right) &\mapsto [g,h] \nonumber
\end{align}
is $R$-bilinear (with respect to these actions). It follows from this definition that any ring of scalars $R$ of $G$ embeds in $End(G/Z(G))$ and in $End(G')$. From this point on we always identify $R$ with its image in $End(G/Z(G))$.

A ring of scalars  $R$ of a 2-nilpotent group $G$ is called a \emph{largest} ring of scalars of $G$ if $K\leq R \leq End(G/Z(G))$ for any other ring of scalars of $G$.  We note that these notions can be formulated for arbitrary nilpotent groups (not necessarily of class $2$). However, since we will not need these in this paper we refer the reader to \cite{Mi_So_2, Mi_So_1} for the general case. 

Our strategy in studying the Diophantine problems in  nilpotent groups $G$ of an arbitrary  nilpotency class $c \geq 2$ is to apply  the largest rings of scalars to the Diophantine problems in  2-step nilpotent groups and then use the result (Proposition \ref{any_step}) that the 2-step nilpotent group $G/\gamma_3(G)$ is e-interpretable in $G$.

\begin{fact}[\cite{GMO1}]\label{t: e-interpret_ring_of_scalars}
The largest ring of scalars of a finitely generated $2$-step nilpotent group $G$ is e-interpretable in $G$.
\end{fact}

We will need the following key concept before applying Theorem \ref{t: e-interpret_ring_of_scalars}. Recall that  $C_G(g)$ denotes the centralizer of $g$ in $G$, i.e., $C_G(g) = \{x \in G \mid xg=gx\}$. 

\begin{dfn}\label{generalPosDfn}
An element $g$ of a group $G$ is said to be   \emph{centralizer-small} (c-small) if $C_G(g)/Z(G)$ is infinite cyclic generated by $gZ(G)$. 
\end{dfn}
Observe, that if $g$ is c-small in $G$ then $C_G(g)=\{ g^t z \ | \ t\in \mbb{Z}, \ z \in Z(G)\},$

\begin{remark}\label{r: exists_b}
Let $G$ be a finitely generated $2$-nilpotent group and let $a\in G$ be a centralizer-small element. Then there exists an element $b \in G$ such that $[a,b]$ has infinite order.
\end{remark}
\begin{proof}
Let $b_1, \dots, b_m$ be a generating set of $G$, and assume that $[a, b_i]$ has finite order $n_i$ for all $i=1,\dots, m$ (otherwise we are done). Writing $n=n_1 \dots n_m$ we have $[a, b_i]^n=1$ for all $i$, which implies that $[a, g]^n=[a^n, g]=1$ for all $g\in G$, and so $a^n \in Z(G)$, contradicting the definition of $c$-small elements.
\end{proof}

\begin{thm}\label{CSmallMaxZ}
Let $G$ be a finitely generated $2$-step nilpotent group having a centralizer-small element.  Then the largest ring of scalars  of $G$ is isomorphic to the ring of integers $\mathbb{Z}$.
\end{thm}
\begin{proof}
Let $a\in G$ be a c-small element and let $b\in G$ be such that $[a,b]$ has infinite order (see Remark \ref{r: exists_b}). Let $\ \bar{{}} \ $ denote the natural projection from $G$ to  $G/Z(G)$.  Suppose $R$ is a ring of scalars of $G$, and fix an element $r\in R$. For $g\in G$ denote by  $g_r$ an arbitrary element in $G$ such that $\bar{g_r}=r\bar{g}$. Following  our notation, let $a_r$ and $b_r$ be elements from $G$ such that $\bar{a_r} = r \bar{a}$ and $\bar{b_r} = r \bar{b}$. 
Then
$$
[a, a_r] = [\bar{a}, r \bar{a}]= r[\bar{a} , \bar{a}]= 1.
$$
Since $a$ is centralizer-small, $a_r=a^{t_{r,a}} z_a$ for some unique $t_{r,a} \in \mbb{Z}$ and some $z_{r,a} \in Z(G)$.  Let $\phi: R \to \mathbb{Z}$  be the map $\phi(r)=t_{r,a}$. We claim that $\phi$ is an isomorphism.

We next show that $r\bar{g}=\bar{g}^{t_{r,a}}$ for all  $g\in G$.  Indeed, taking any  $g\in G$,
\begin{equation}\label{e: 29march}
[g^{t_{r,a}}, a]=[g, a^{t_{r,a}}] = [\bar{g}, r\bar{a}] = [r\bar{g}, \bar{a}] = [ g_r, a].
\end{equation}
Hence $a$ commutes with $g_r g^{-t_{r,a}}$, and since $a$ is centralizer-small we obtain $g_r = g^{t_{r,a}} a^{s_{r,g}} z_{r,g}$ for some integer $s_{r,g}$ and some $z_{r,g}\in Z(G)$. In particular, $b_r = b^{t_{r,a}}a^{s_{r,b}} z_{r,b}$. We obtain $$[b,a]^{s_{r,b}}=[b, b_rb^{-t_{r,a}}] = [b, b_r] = [\bar{b}, r\bar{b}]=r[\bar{b}, \bar{b}]=1,$$ hence $s_{r,b}=0$ and $b_r = b^{t_{r,a}}z_{b,r}$ for some $z_{b,r}\in Z(G)$.
Now, for an arbitrary $g\in G$,
$$
[g_r, b]= [r \bar{g}, \bar{b}] = [ \bar{g}, r\bar{b}] = [ \bar{g}, \bar{b}^{t_{r,a}}]  = [g^{t_{r,a}}, b]
$$
from where we obtain $[g_r^{-1} g^{t_{r,a}}, b] = 1$, but  since  $g_r = g^{t_{r,a}} a^{s_{r,g}} z_{r,g}$, it follows that
$$
[a^{-s_{r,g}}, b] = [a,b]^{-s_{r,g}} = 1,
$$
which implies $s_{r,g}=0$. Thus $r \bar{g} = \bar{g}^{t_{r,a}}$ for all $g\in G$. From now on we denote $t_{r,a}$ simply by $t_r$.

 We next show that $\phi$ is a ring homomorphism. To clarify our arguments we will write $\oplus$, $\odot $, $1_R$ and $0_R$ when referring to  addition and multiplication in $R$, and to its identity and zero element, respectively. Recall, that $R$ is viewed as a subring of $\End(G/Z(G))$. Hence $\oplus$ is addition of endomorphisms, $\odot$ is composition, $1_R$ is the identity endomorphism, and $0_R$ is the endomorphism that sends all elements to the identity element of $G/Z(G)$.  For all other structures we use standard notation. 
 
 First note that $$\bar{a}^{t_{r_1 \oplus  r_2}}=(r_1\oplus  r_2)\bar{a}=\left(r_1\bar{a}\right)\left(r_2\bar{a}\right)= \bar{a}^{t_{r_1}} \bar{a}^{t_{r_2}}=\bar{a}^{t_{r_1}+t_{r_2}}.$$ Since $\bar{a}$ has infinite order in $G/Z(G)$,  one has 
 $$
 \phi(r_1\oplus  r_2)=t_{r_1\oplus  r_2}=t_{r_1} + t_{r_2}=\phi(r_1)+\phi(r_2). 
 $$
 Similarly with the multiplication operation,
$$
	\bar{a}^{t_{r_1\odot  r_2}}= (r_1\odot  r_2)\bar{a}=r_1(r_2(\bar{a}))=r_1\left(\bar{a}^{t_{r_2}}\right)=\bar{a}^{t_{r_1}t_{r_2}}.
$$
 Again,  since $\bar{a}$ has infinite order in $G/Z(G)$, 
 $$
 \phi(r_1\odot  r_2)=t_{r_1\odot  r_2}=t_{r_1}t_{r_2}=\phi(r_1)\phi(r_2).
 $$

 Notice also that $\bar{a}^{(t_{1_R})}= (1_R) \bar{a} = \bar{a}$, because since $R$ is (identified with) a subring of $\End(G/Z(G))$, the identity element of $R$ is the identity endomorphism of $G/Z(G)$. This implies that $\phi(1_R)=t_{1_R}=1$, completing the proof that $\phi$ is a ring homomorphism.

Moreover, $\phi$ is surjective, because given $k\in \mathbb{Z}$, we have 
$$\phi\left(\sum_{i=1}^k 1_R \right) = \sum_{i=1}^ k \phi(1_R) = k.$$ 
Finally, notice that if $\phi(r)=t_r=0$ for some $r$, then $r\bar{g}= \bar{1}$ for all $g\in G$. Since $R$ acts faithfully on $G/Z(G)$   $r = 0$ in $R$, i.e., \ $r=0_R$.  We conclude that $\phi$ is a ring isomorphism. This proves that any non-zero ring of scalars of $G$  is isomorphic to $\mathbb{Z}$. In particular, this is true of the largest ring of scalars of $G$.
\end{proof}

The following result is an immediate consequence of the previous two theorems and of Corollaries \ref{RedCor} and \ref{transitivity}.

\begin{thm}\label{mainThmEq2}
    Let $G$ be a finitely generated $2$-step nilpotent group  with a centralizer-small element.  Then the ring $\mathbb{Z}$ is e-interpretable in $G$, and the Diophantine problem in $G$ is undecidable.  
\end{thm}

We recover one  of the results from \cite{Duchin2}:

\begin{cor}\label{undecinfreenilp}
The ring $\mbb{Z}$ is  e-interpretable in any non-abelian free nilpotent group $N$. Consequently, $\mc{D}(\mbb{Z})$ is reducible to $\mc{D}(N)$, and $\mc{D}(N)$ is undecidable.
\end{cor}

\begin{proof}
Suppose first that $N$ is a non-abelian $2$-step free nilpotent group freely generated by $A=\{ a_1, \dots, a_n\}$. Let $C=\{[a_i, a_j] \mid 1\leq i< j \leq n\}$. Using that $(A; C)$ is a Malcev basis of $N$ and `bilinearity' of $[\cdot, \cdot]$ (see \eqref{e: multilinear}) one obtains that $a_1$ is a centralizer-small element of $N$. Hence the ring $\mbb{Z}$ is  e-interpretable in $N$,  by  Theorem \ref{mainThmEq2}.

Now assume that $N$ is $c$-step nilpotent for some $c>2$. Then $N/\gamma_3(N)$ is a non-abelian $2$-step free nilpotent group, and by the previous paragraph, the ring $\mbb{Z}$ is  e-interpretable in $N/\gamma_3(N)$. By Proposition \ref{any_step}, $N/\gamma_3(N)$ is  e-interpretable in $N$. Finally, since e-interpretability is a transitive property, 
$\mbb{Z}$ is  e-interpretable in $N$. Hence, by   Corollary \ref{RedCor}, $\mc{D}(\mbb{Z})$ is reducible to $\mc{D}(N)$. This makes $\mc{D}(N)$ undecidable.
\end{proof}

We finish this section by providing another application of the largest ring of scalars regarding direct decomposability. We note that this result will not be used further in the text.
\begin{prop}\label{dirIndec}
Suppose $\mathbb{Z}$ is the largest ring of scalars  of a finitely generated  torsion-free $2$-step nilpotent group $G$. Then $G$ cannot be decomposed into a direct product of  non-abelian subgroups. 
\end{prop}
\begin{proof}
Suppose to the contrary that $G= H \times K$ for some  non-abelian subgroups $H,K$ of $G$. Then $H$ and $K$ are non-abelian finitely generated torsion-free $2$-step nilpotent groups. It is immediate to verify that $Z(G)=Z(H)\times Z(K)$,  that  $G/Z(G) \cong  H/Z(H) \times K/Z(K)$, and that  $G'=H' \times K'$.  

Consider the natural ring actions  of $\mathbb{Z}^2$ on $H/Z(H) \times K/Z(K)$ and on $H' \times K'$ defined by component-wise exponentiation (or component-wise multiplication if one is using additive notation):
\begin{equation}\label{123}
(r_1, r_2)(h, k)= (h^{r_1}, k^{r_2}),
\end{equation}
for  $(r_1, r_2)\in \mathbb{Z}^2$ and  $(h, k)$ in $ H/Z(H) \times K/Z(K)$ or in $H' \times K'$. Note that, by fixing a tuple $(r_1, r_2)$ in \eqref{123} we obtain an endomorphism of $H/Z(H) \times K/Z(K)$ (or $H'\times K'$).  

We next show that these actions are faithful. Suppose $$(r_1,r_2)(hZ(H),kZ(K))=(Z(H),Z(K))$$ for all $h, k\in H, K$.  Then $h^{r_1} \in Z(H)$ and $k^{r_2}\in Z(K)$ for all $h,k\in H,K$. Since both $H$ and $K$ are torsion-free and $2$-step nilpotent, $H/Z(H)$ and $K/Z(K)$ are free abelian, and so either $H=Z(H)$ or $r_1=0$, and similarly for $K$ and $r_2$. Since $H$ and $K$ are non-abelian, $r_1=r_2=0$. This shows that the action \eqref{123} is faithful, which implies  there exists an embedding of rings $\mbb{Z}^2 \hookrightarrow \End\left(H/Z(H) \times K/Z(K)\right)$.  Similar arguments show that $\mbb{Z}^2$ also embeds  (as a ring) into $\End(H' \times K')$.

Moreover,
$$
[(r_1, r_2)(u_1, u_2), (v_1, v_2)]=[(u_1^{r_1}, u_2^{r_2}), (v_1, v_2)]=\left([u_1^{r_1}, v_1],[u_2^{r_2},v_2]\right) =$$$$=\left([u_1,v_1^{r_1}], [u_2, v_2^{r_2}]\right)=[(u_1, u_2), (r_1, r_2)(v_1, v_2)].
$$
For all $(r_1, r_2) \in \mbb{Z}^{2}$ and all $(u_1, u_2), (v_1, v_2) \in H/Z(H) \times K/Z(K)$. Similarly:
$$
[(r_1, r_2)(u_1, u_2), (v_1, v_2)]=\left([u_1^{r_1}, v_1],[u_2^{r_2},v_2]\right) =$$$$=\left([u_1,v_1]^{r_1}, [u_2, v_2]^{r_2}\right)=(r_1, r_2)\left( [u_1, v_1], [u_2, v_2]\right) = (r_1, r_2)[(u_1, u_2), (v_1, v_2)].
$$
Hence  $[\mb{r}\mb{u},\mb{v}]=[\mb{u},\mb{r}\mb{v}]=\mb{r}[\mb{u}, \mb{v}]$ for all $\mb{r}\in \mathbb{Z}^2$ and all $\mb{u},\mb{v}\in H/Z(H) \times K/Z(K)$. 
Thus, $\mathbb{Z}^2$ is a ring of scalars of $G=H \times K$. By  definition, $\mathbb{Z}^2$ embeds into the largest ring of scalars of $G$, which is $\mathbb{Z}$ by hypothesis - a contradiction.
\end{proof}

\section{Nilpotent groups given by full rank presentations}

Throughout this section $A = \{a_1, \ldots,a_n\}$,   $R=\{r_1, \dots, r_m\}$ is  a set of $m$ words in the alphabet $A^{\pm 1}$,  and $M(A, R)$ is the relation matrix of the group presentation $\langle A \mid R\rangle_{\mc{N}_c}$. In this section we will study the algebraic structure of groups with presentation 
\begin{equation}\label{e: main_presentation} 
G= \langle A \mid R\rangle_{\mc{N}_c},
\end{equation} 
such that $M(A, R)$ has full rank (i.e. the rank of $M(A,R)$ is equal to $\min(m,n)$). In this case we say that \eqref{e: main_presentation} is a \emph{full rank} presentation.

We adopt the following standard convention regarding the projection of elements onto factor groups: Suppose  $K$ has been obtained from another group $H$ by adding some relations, and let $\pi:H\to K$ be the natural projection of $H$ onto $K$. To avoid constantly referring to $\pi$, we will often speak of  \emph{elements $h$ from $H$ seen in  $K$} (or \emph{projected onto $K$}), rather than of elements $\pi(h)$. Similarly, for $h_1, h_2\in H$, instead of  writing $\pi(h_1) = \pi(h_2)$, we will say that  \emph{$h_1=h_2$ in $K$}.

\subsection{Structure of full rank nilpotent groups}
In this subsection we  describe the structure of full rank nilpotent groups.

We shall say that two subsets $S_1, S_2$ of a free abelian group are \emph{linearly independent} if they generate subgroups that have trivial intersection. Consider a $c$-step nilpotent presentation $\langle A\mid R\rangle_{{\mathcal N}_c}$, and let $A_0=\{a_{i_1}, \dots, a_{i_k}\}$ be a subset of $A$. Let $\pi$ be the projection from the free group freely generated by $A$ onto the free abelian group freely generated by $A$. We say that $A_0$ is \emph{linearly independent} from $R$ (in the abelianization) if the sets $\pi(A_0)$ anf $\pi(R)$ are linearly independent.

\begin{thm}\label{A}
Let $G$ be a finitely generated nilpotent group of class $c \geq 1$ given by a finite full rank presentation $G = \langle A \mid R\rangle_{{\mathcal N}_c} = \langle a_1, \ldots,a_n \mid  r_1, \ldots,r_m \rangle_{{\mathcal N}_c}$.
 Then  there exists a subset $A_0\subseteq A$ with $|A_0| = n-m$ ($A_0=\emptyset$ if $m\geq n$) such that the following holds:
\begin{enumerate}
\item If $m\geq n$, then $G$ is finite.
\item If $m=n-1$, then $\langle A_0\rangle$ is infinite cyclic and has finite index in $G$.%
\item If $m \leq n-2$, then $\langle A_0 \rangle$ is a free $c$-step nilpotent subgroup of rank $n-m$ which has finite index in $G$. 
\end{enumerate}
 Furthermore, $A_0$ can be chosen to be  any subset $\{a_{i_1}, \ldots,a_{i_{n-m}}\}$ of $A$ such that the rank of the matrix $M(A,R)$ coincides with the rank of the matrix obtained from $M(A,R)$ after removing its $i_1, \dots, i_{n-m}$-th columns.

\end{thm}

Before proving Theorem \ref{A} we introduce an auxiliary result:

\begin{lemma}\label{r: independence}
Let $M(A,R)$ be the relation matrix of a presentation $\langle A \mid R\rangle_{\mc{N}_c}$ and  $A_0=\{a_{i_1}, \dots, a_{i_{n-m}}\}$ be some subset of $A$. Then  the subgroups $\langle R\rangle$ and $\langle A_0\rangle$ are linearly independent in $N/N'$ if and only if  the rank of the matrix $M(A,R)$ coincides with the rank of the matrix obtained from $M(A,R)$ after removing its $i_1, \dots, i_{n-m}$-th columns.
\end{lemma}
\begin{proof}
Consider the homomorphism $\theta$ from the free abelian group generated by $A$ to the free abelian group generated by $A\smallsetminus A_0$ that sends all elements of $A_0$ to $1$ and fixes all elements of $A\smallsetminus A_0$. Then $A_0$ is linearly independent from $R$ if and only if $\theta$ is an injection when restricted to the subgroup generated by $R$, which occurs if and only if the subgroup generated by $R$ has the same rank (minimum number of generators) as its image under $\theta$. But the latter holds if and only if the rank of $M(A,R)$ is the same as the rank of the matrix obtained from $M(A,R)$ by removing its $i_1, \dots, i_{n-m}$-th columns (the columns corresponding to the elements $A_0$), and so the lemma follows.
\end{proof}

\begin{proof}[Proof of Theorem \ref{A}]

Assume first that $m\leq n$, in which case the rank of $M(A,R)$ is $m$ by hypothesis. In this case,  there exist $m$ integers $1\leq j_1, \dots, j_m \leq m$ such that the $m\times m$ matrix formed by the columns $j_1, \dots, j_m$ has rank $m$.  Let $i_1, \dots, i_{n-m}$ be the complement of $\{j_1, \dots, j_m\}$ in the set $\{1, 2,\dots, m\}$. Then by Lemma \ref{r: independence} the subset $A_0=\{a_{i_1}, \dots, a_{i_{n-m}}\}$ is linearly independent from $R$. If $n=m$ then we convene $A_0=\emptyset$.   Without loss of generality we can suppose $A_0 = \{a_{m+1},\dots, a_n\}$.   We denote $H=\langle A_0, R\rangle$.

For the next arguments we follow the notation and terminology regarding $\mbb{Q}$-completions which was introduced in Subsection \ref{s: Q_completions}. Let $N=N(A)=\mathcal{N}_c(a_1,\ldots,a_n)$, and let $\phi: N \to \ol{N}$ be its $\mbb{Q}$-completion. 
Since $N$ is torsion-free, we identify $N$ with the  subgroup $\phi(N)\leq \ol{N}$ isomorphic to $N$ (see Property 2 in Fact \ref{p: properties_Q_completions}), hence we will omit the notation $\phi$ from now on. Denote by $\pi$ the natural projection $\pi: N\to N/N'$. As already mentioned, in $N/N'$,
 $\pi(A_0\cup R)$ forms a set of $n$ linearly independent elements, and so $\pi(H)=\pi(\langle A_0\cup R\rangle)$ has finite index in $N/N'\cong \mbb{Z}^n$. Hence $\ol{\pi(H)}=\ol{(N/N')}$ and  $\ol{(N/N')}$ is $\mbb{Q}$-generated by  $\pi(A_0\cup R)$ (Property 1 in Fact \ref{p: properties_Q_completions}), where  again we have identified $N/N'$ with $\ol{(N/N')}$. Denoting by $\rho$ the natural projection of $\ol{N}$ onto its abelianization $\ol{N}/(\ol{N})'$ we have from our notation assumptions and  \cite[Corollary 8.14]{Warfield} that the map $\pi(a_i) \mapsto \rho(a_i)$ induces an isomorphism between $\ol{(N/N)'}$ and $\ol{N}/(\ol{N})'$. Hence $\rho(R\cup A_0)$ $\mbb{Q}$-generates the abelianization of $\ol{N}$.
 In particular $\langle \rho(A_0 \cup R)\rangle_{\mbb{Q}} = \rho(\langle A_0\cup R\rangle_{\mbb{Q}}) = \rho(\ol{H})$ generates (as a group) the abelianization of $\ol{N}$, which means that $\ol{H}$ generates $\ol{N}$ (again as a group ---this is due to the well-known fact that a generating set of the abelianization of a nilpotent group lifts to a generating set of the group). We conclude that $A_0 \cup R$ $\mbb{Q}$-generates $\ol{N}$. 
Since $ |A_0\cup R| = n$ and $\ol{N}$ is a free nilpotent $\mbb{Q}$-group of rank $n$, this generation is  free. 

Now, from Fact  \ref{p: properties_Q_completions} and the above observations we have  \begin{align*}\ol{G}&\cong \ol{(N(A)/\langle\langle R\rangle\rangle)}\cong \ol{N(A)}/\ol{\langle\langle R\rangle\rangle} = \langle A_0 \cup R\rangle_{\mbb{Q}}/\langle\langle R\rangle\rangle_{\mbb{Q}} =\\  &=\ol{\langle A_0 \cup R\rangle} /\ol{\langle \langle R \rangle \rangle}\cong \ol{(\langle A_0 \cup R\rangle /\langle \langle R \rangle \rangle)}\cong \ol{N(A_0)}.\end{align*} Moreover  the resulting isomorphism $\ol{G} \cong \ol{N(A_0)}$ acts as the identity on the set $A_0$, and so $\ol{G}$ is a free nilpotent $\mbb{Q}$-group freely $\mbb{Q}$-generated by $A_0$. Furthermore, from Properties 4 and 5 in Fact \ref{p: properties_Q_completions} we have that $\langle A_0 \rangle$ has finite index in $G$ and that the subgroup $\langle A_0\rangle \leq G$ is  free nilpotent. This completes the proof of the theorem for the case $n\leq m$. 

The case $n> m$ follows from this result since then $G$ is a quotient of a group with full rank presentation of deficiency $0$ (corresponding to the case $n=m$ above).
\end{proof}

\subsection{Presentations and Smith normal forms}

In this subsection we study  the matrix $M(A, R)$ in relation to a group presentation $G= \langle A \mid R\rangle_{\mc{V}}$, where  $\mc{V}$ is any variety of groups. We show that $M(A, R)$ can be assumed to have a very simple form, a fact that will be  helpful when studying full rank $2$-step nilpotent groups.  

Recall that an integer matrix is said to be in Smith normal form if it has  zeroes everywhere except in the diagonal, which consists first of a possibly empty sequence of  nonzero integers and it is followed by a possibly empty sequence of zeroes. Observe that the relation matrix of a presentation in a variety $\mc{V}$, $G=\langle a_1, \dots, a_n \mid r_1, \dots, r_m\rangle_{\mc{V}}$, is in Smith normal form if and only if $r_i = a_i^{\alpha_i}c_i$ where $(\alpha_1, \dots, \alpha_m)$ is the diagonal of $M(A,R)$, and  the $c_i$ are elements from  $ G'$ for all $i=1,\dots, m$. 
%

\begin{lemma}\label{SmithLemma2}
Let $G$ be a group with presentation $\langle A \mid R\rangle_{\mc{V}}$ in a variety $\mc{V}$. 
Then there exist a group presentation $\langle A' \mid R'\rangle_{\mc{V}}$ with the following properties: 1) $G$ is isomorphic to the group defined by $\langle A' \mid R'\rangle_{\mc{V}}$, 2) $M(A', R')$ is in Smith normal form, 3) $|A| = |A'|$, $|R|=|R'|$, and $rank(M(A,R))=rank(M(A',R'))$.
%
%
\end{lemma}
\begin{proof} 
First suppose that $A'$ has been obtained from $A$ by  performing a Nielsen transformation $a_j \mapsto a_j'=a_i^{\pm 1}a_j$. 
Then $A'$ freely generates $N$, and $G \cong N(A)/\langle \langle R \rangle \rangle \cong N(A')/
\langle \langle R'\rangle \rangle$, where $R'$ denotes the set of words from $R$ rewritten as words on the alphabet $A'$. A similar situation holds when performing a Nielsen transformation on $R$. Hence, we may apply as many Nielsen transformations to $A$ and $R$ as we wish, without changing the isomorphism class of the group $G$.

On the other hand, recall that the Smith normal form of any integer matrix, in particular $M(A, R)$, can be obtained by successively adding or subtracting a row of $M(A,R)$ to another different row, or a column to another different column, and by reordering rows and columns. Of course, this does not change the rank of the matrix. It is an easy exercise to check that, for each matrix $M'$ obtained from $M(A,R)$ by applying one of these operations, one may perform one Nielsen transformations to $A$ or $R$ so that, for the resulting sets $A', R'$, the matrix $M(A', R')$ is precisely $M'$. More concretely, matrix row operations correspond to Nielsen transformations on $R$, and matrix column operations correspond to Nielsen transformations on $A$.  The lemma then follows from this remark and from the observations made in the previous paragraph.
\end{proof}

\begin{prop}\label{p: trivial_torsion_in_abelianization}
Let $G$ be a finitely generated nilpotent group of class $c \geq 2$ given by a finite full rank presentation $G = \langle a_1, \ldots,a_n \mid  r_1, \ldots,r_m \rangle_{\mc{N}_c}$. Assume that the abelianization $G/G'$ of $G$ has trivial torsion subgroup. Then $G$ is trivial if $m \geq n$; it is infinite cyclic if $m=n-1$; and it is free nilpotent of class $c$ and rank $n-m$ if $m\leq n-2$.
\end{prop}
\begin{proof}
By Lemma \ref{SmithLemma2} we can assume that $r_i = a^{n_i}c_i$ with $\alpha_i \in \mbb{Z}\setminus\{0\}, c_i \in G'$,  for all $i=1, \dots, m$. In this case $G/G'$ has presentation $\langle a_1, \dots, a_n \mid a_1^{\alpha_1} = 1, \dots, a_m^{\alpha_m}= 1\rangle_{\mc{N}_1}$ in the variety of abelian groups. From the hypothesis we have $a_1 = \dots = a_m = 1$ in $G/G'$, and so $G/G'$ is generated by $a_{m+1}, \dots, a_n$. A well-known property of nilpotent groups (Magnus lifting theorem) states that any generating set of $G/G'$ lifts to a generating set of  $G$, hence $G$ is generated by $a_{m+1}, \dots, a_n$. In particular, if $G/G'$ is trivial then so is $G$. The proposition is now a consequence of Theorem \ref{A}.
\end{proof}

\subsection{Further structural results in nilpotency class $2$}

In this subsection we prove some more results regarding the structure of $2$-step nilpotent groups given by a full rank presentation. We follow the notation of the previous section, assuming this time that $c=2$. Hence we assume that $G$ is a $2$-step nilpotent group given by a full rank presentation $G=\langle A \mid R \rangle_{\mc{N}_2}$. Due to Lemma \ref{SmithLemma2} we assume that $M(A, R)$ is in Smith normal form, which is equivalent to saying that $R=\{r_1, \dots, r_m\}$ with $r_i = a_i^{\alpha_i} c_i$ for some nonzero integers $\alpha_i$ and some elements $c_i$ from the commutator subgroup $G'$, for all $i=1, \dots, m$.
Recall that the commutator operation $[\cdot, \cdot]$ behaves "bilinearly" in $2$-step nilpotent groups (see \eqref{e: multilinear}). We shall use this fact without further reference.

The following is a technical lemma that will be used later for studying the center of $G$, as well as  centralizers of some of its generators.

\begin{lemma}\label{keyLemma_new}
Let $G$ be a finitely generated nilpotent group of class $c=2$ given by a finite full rank presentation $G = \langle A \mid  R \rangle_{\mc{N}_2} =\langle a_1, \ldots,a_n \mid  r_1, \ldots,r_m \rangle_{\mc{N}_2}$, with $m\leq n$ and with $M(A, R)$ in Smith normal form. Let $h$ be an element from $N$ of the form $$h= a_{m+1}^{\gamma_{m+1}} \dots a_n^{\gamma_n} c,$$ with   $c\in N'$ and $\gamma_i\in\mathbb{Z}$ for all $i$. Denote $A_{[1 : m]}^{\alpha} = \{a_1^{\alpha_1}, \dots, a_m^{\alpha_m}\}$ and let $$[N, A_{[1 : m]}^{\alpha}] = \langle \{[x,y]\mid x \in N,\ y \in A_{[1 : m]}^{\alpha}\} \rangle.$$ Assume the image of  $h$ in $G$ is $1$. Then $\gamma_{m+1}=\ldots=\gamma_n=0$, and $c\in  [N, A_{[1 : m]}^{\alpha}]$ in $N$.
\end{lemma} 
\begin{proof}
Since $h=1$ in $G$, $h$ can be written in $N$ as  a product of conjugates of elements from $R$ and their inverses.  Hence, there exists a sequence of elements $w_j \in N$, signs $\e_{j} \in \{+1, -1\}$, and relators $$r_{i_j} = a_{i_j}^{\alpha_{i_j}}c_{i_j} \in R=\{r_1, \dots, r_m\}, \quad (j=1, \dots, p),$$ such that, in $N$,
\begin{equation}\label{keyEqn}
h=a_{m+1}^{\gamma_{m+1}} \dots a_n^{\gamma_n} c  = \prod_{j=1}^p w_{j}^{-1} a_{i_j}^{\e_{j} \alpha_{i_j}} c_{i_j}^{\e_{j}} w_{j}
\end{equation}
We now wish to write the Malcev representation  of the right hand-side of \eqref{keyEqn}. To do so, it suffices to move the $a_{i_j}^{\epsilon_j \alpha_{i_j}}$ to the left by repeatedly applying the identity: $xy = yx [x,y]$. Once this is done, we move all commutators introduced this way, and all the $c_{i_j}$'s (which belong to the center of $N$ since $c=2$) to the left of \eqref{keyEqn}. The elements $w_j^{-1}$ and $w_j$ then cancel between themselves. Notice that during these operations we only introduce commutators from $[N, A_{[1: m]}^{\alpha}]$. For  $k=1, \dots, m$, let  $\l_k$ be the sum of those $\e_i$'s for which $i_j=k$, i.e.\
$$
\l_k = \sum_{\substack{i_j = k}} \e_{j}.
$$
Then  \eqref{keyEqn} can be written as
\begin{equation}\label{keyEqn3}
h=a_{m+1}^{\gamma_{m+1}} \dots a_n^{\gamma_n} c
= \left(\prod_{k=1}^{m} a_k^{\l_k \alpha_k} \prod_{k=1}^{m} c_{k}^{\l_k} \right) d,
\end{equation}
where $d \in [N, A_{[1:m]}^{\alpha}]$.
This equality takes place in the free $2$-step nilpotent group $N$. By uniqueness of Malcev coordinates (see \eqref{malcev_rep}), and since $\alpha_k\neq 0$ for all $k$, we have $\l_k=0$ for all $k\leq m$, and $\gamma_k=0$ for all $k\geq m+1$. It follows that, in $N$, $h=c=d \in  [N, A_{[1:m]}^{\alpha}]$, as needed. 
\end{proof}

It immediately follows:
\begin{cor}\label{c: infinite_order}
	Let $G$ be a finitely generated nilpotent group of class $2$ given by a finite full rank presentation $G = \langle A \mid  R \rangle_{\mc{N}_2}= \langle a_1, \ldots,a_n \mid  r_1, \ldots,r_m \rangle_{\mc{N}_2}$, with $m\leq n-1$ and with $M(A, R)$ in Smith normal form. Then $a_t$ has infinite order for all $t \geq m+1$, and moreover $[a_t, a_{t'}]$ has infinite order as well for any two $m+1\leq t \neq t' \leq n$.
\end{cor}

We are ready to describe the centralizers of some generators of $G$.

\begin{lemma}\label{l: centralizers}
Let $G$ be a finitely generated nilpotent group of class $2$ given by a finite full rank presentation $G = \langle a_1, \ldots,a_n \mid  r_1, \ldots,r_m \rangle_{\mc{N}_2}$, with $m\leq n - 1$ and with $M(A, R)$ in Smith normal form. Then for all $m+1\leq t \leq n$ and any element $h \in N$,  the images of $h$ and $a_t$ commute in $G$ if and only if 
$$
	h=a_t^{\gamma_t}\left(\prod_{i=1}^m a_i^{\alpha_i \beta_i}\right) c \quad \text{(in } N \text{)} 
$$
for some $\gamma_t\in \mbb{Z}$, $\beta_i \in \mbb{Z}$ ($i=1, \dots, m$) and  $c\in N'$.
Consequently, 
$$
C_G(a_t) = \langle a_t \rangle G'
$$  
for all $m+1\leq t \leq n$, where $C_G(a_t)$ denotes the centralizer of $a_t$ in $G$.
\end{lemma}
\begin{proof}
Let $t$ be such that $m+1\leq t \leq n$ and let $h\in N$ satisfy $[h,a_t] = 1$ in $G$. Write $h= \prod_{i= 1}^{n} a_i^{\lambda_i} c$ for some $\lambda_i \in \mbb{Z}$ and $c \in N'$. Then, in $N$,
$$
[h,a_t]= \prod_{\substack{i=1\\ i\neq t}}^{n} [a_i, a_t]^{\lambda_i} 
$$
Moreover, since $[h,a_t]=1$ in $G$,  Lemma \ref{keyLemma_new} ensures that $[h,a_t] \in \langle [N, A_{[1: m]}^{\alpha}]\rangle$ in $N$.  Using that the commutator operation behaves `biliniearly' in $2$-step nilpotent groups we obtain the existence of integers $\mu_{j}$ and elements $w_j \in N$ such that, in $N$,
$$
[h, a_t] = \prod_{\substack{i=1\\ i\neq t}}^{n} [a_i, a_t]^{\lambda_i} =\prod_{1\leq j \leq m} [w_j, a_j]^{\alpha_{j} \mu_{j}}.
$$
The Malcev representation of the right-hand side of the above equation uses only basic commutators of the form $[a_{s}, a_{s'}]^{\alpha_{*}}$ with either $s$ or $s'$ being at most $m$, and $\alpha_* \in \{\alpha_s, \alpha_{s'}\}$. It follows by unicity of Malcev coordinates and from $t\geq m+1$ that $\lambda_{i}=0$ for all $i\geq m+1$ with $i\neq t$, and  that $\alpha_j$ divides $\lambda_j$ for all $1\leq j\leq m$ with $j\neq t$. Hence $h =a_t^{\lambda_t}h_0$, where $h_0=\prod_{i=1}^{m} a_i^{\alpha_i \beta_i} c $ for some $\beta_i \in \mbb{Z}$, and thus $h$ has the desired form. Moreover, $h_0$ belongs to $G'$ in $G$, hence $h_0 \in \langle a_t \rangle G'$ in $G$. It follows that  $C_G(a_t) \leq \langle a_t \rangle G'$. The reverse inclusion is immediate since $G'\leq Z(G)$. 
\end{proof}

\begin{thm}\label{t: center}
Let $G$ be a finitely generated nilpotent group of class $2$ given by a finite full rank presentation $G = \langle a_1, \ldots,a_n \mid  r_1, \ldots,r_m \rangle_{\mc{N}_2}$, and where $m\leq n - 2$. Then $Z(G) = G'$.
\end{thm}
\begin{proof}
By Lemma \ref{SmithLemma2} we can assume without loss of generality that $M(A, R)$ is in Smith normal form. Suppose $m\leq n-2$, and let $\pi: G \to G/G'$ be the canonical projection of $G$ onto $G/G'$. Then, by the previous Lemma \ref{l: centralizers}, 
\begin{align*}
\pi(C_G(a_{n-1})\cap& C_G(a_n)) \leq \pi(C_G(a_{n-1})) \cap \pi(C_G(a_n))\\ = &\langle \pi(a_{n-1}) \rangle \cap \langle \pi(a_n) \rangle = 1,
\end{align*}
where the last equality is due to the fact that the torsion-free part of $G/G'$ is freely generated by $\pi(a_{m+1}), \dots, \pi(a_n)$, which is apparent from the presentation of $G$.

Therefore $Z(G) \leq C_G(a_{n-1}) \cap C_G(a_n) \leq G'$. The reverse inclusion is due to the fact that $G$ is of nilpotency class $2$.
\end{proof}

\begin{thm}\label{t: Torsion_in_Random_step2}
Let $G$ be a finitely generated nilpotent group of class $2$ given by a finite full rank presentation $G = \langle a_1, \ldots,a_n \mid  r_1, \ldots,r_m \rangle_{\mc{N}_2}$,  where $ m\leq n - 1$. Then the following are equivalent: 
\begin{enumerate}
    \item $G/G'$ has trivial torsion subgroup.
    \item $\gamma_2(G)$ has trivial torsion subgroup.
    \item $G$ has trivial torsion subgroup.
    \item $G$ is free nilpotent.
\end{enumerate}
\end{thm}
\begin{proof}
We prove that $1$ implies $2,3$ and $4$, and that if $1$ does not hold then neither $2,3$, nor $4$ hold.

Assume that $M(A, R)$ is in Smith normal form (due to Lemma \ref{SmithLemma2}). Suppose $G/G'$ has non-trivial torsion subgroup. Then at least one of the exponents $\alpha_i$, say $\alpha_1$, is neither $1$ or $-1$ in $G$. We have $[a_1, a_n]^{\alpha_1} = [a_1^{\alpha_1}, a_n] = [c_1, a_n] = 1$. Hence $[a_1, a_n]$ is a torsion element of $\gamma_2(G)$  provided that $a_1$ does not commute with $a_n$ in $G$. The latter follows from Lemma \ref{l: centralizers} and unicity of Malcev coordinates. Hence $\gamma_2(G)$ and $G$ have non-trivial torsion subgroups, and $G$ is not free nilpotent. 

If, on the other hand, $G/G'$ has trivial torsion subgroup then $G$ is free nilpotent due to Proposition \ref{p: trivial_torsion_in_abelianization}, hence torsion-free.
\end{proof}

\begin{question}
Is the statement of Theorem \ref{t: Torsion_in_Random_step2} true when the nilpotency class of $G$ is greater than $2$ (replacing $\gamma_2(G)$ by $\gamma_c(G)$)?
\end{question}

\subsection{Regularity, first-order rigidity, and QFA}

Recall that ${\it Is}(H)$, the \emph{isolator} of a subgroup $H\leq G$, is defined as $${\it Is}(H) = \left\{g \in G \mid g^n \in H \ \hbox{for some} \ n\in \mathbb{Z},\ n\neq 0 \right\}.$$ We say that $G$ is \emph{regular} if $Z(G) \leq {\it Is}(G')$. We will need the following observation:
\begin{remark}\label{r: from_s_to_2}
Let $G$ be a $c$-step nilpotent group with full rank presentation $G= \langle A\mid R\rangle_{\mc{N}_c}$. Then $\langle A\mid R\rangle_{\mc{N}_2}$ is a full rank presentation of $G/\gamma_3(G)$.
\end{remark}
 
\begin{thm}\label{AAA}
Let $G$ be a finitely generated nilpotent group of class $c \geq 2$ given by a finite full rank presentation $G = \langle a_1, \ldots,a_n \mid  r_1, \ldots,r_m \rangle_{\mc{N}_c}$. 
If $m \leq n-2$  then  $Z(G)\leq G'$. In particular $G$ is regular.
\end{thm}

\begin{proof}
Let $x\in G$ be such that $x\in Z(G)$. Then  $x\gamma_3(G) \in Z(G/\gamma_{3}(G))$. By Remark \ref{r: from_s_to_2} and Theorem \ref{t: center} we have $Z(G/\gamma_{3}(G))\leq (G/\gamma_{3}(G))/(G/\gamma_{3}(G))'$. Now $x \gamma_3(G)\in (G/\gamma_{3}(G))'$. But then $x \in G' \gamma_3(G) \leq G'$, as needed. 
\end{proof}

\begin{thm} \label{th:FAQ}
Let $G$ be a finitely generated nilpotent group of class $c \geq 2$ given by a finite full rank presentation $G = \langle a_1, \ldots,a_n \mid  r_1, \ldots,r_m \rangle_{\mc{N}_c}$. 
If $m \leq n-2$  then  $G$ is QFA, in particular it is first-order rigid.
\end{thm}
\begin{proof}
It is known \cite{Oger-Sabbagh} that any finitely generated regular nilpotent group is QFA. So the result follows from Theorem \ref{AAA}.
\end{proof}

\subsection{Diophantine problem in nilpotent groups of full rank}

We now obtain:
\begin{thm}\label{D}
Let $G$ be a finitely generated nilpotent group of class $c \geq 2$ given by a finite full rank presentation $G = \langle a_1, \ldots,a_n \mid  r_1, \ldots,r_m \rangle_{\mc{N}_c}$.  Then the following holds: 
\begin{enumerate}
\item If $m \leq n-2$, then the ring $\mathbb{Z}$ is e-interpretable in $G$ and the Diophantine problem in $G$ is undecidable.
\item If $m \geq n-1$ then the Diophantine problem in $G$ is decidable. %
\end{enumerate}
\end{thm}

\begin{proof}
Item 2 follows from Theorem \ref{A} and the fact that the Diophantine problem in  virtually abelian groups is decidable \cite{Ershov}. Assume therefore that $m\leq n-2$. By Lemma \ref{SmithLemma2} we assume without loss of generality that $M(A, R)$ is in Smith normal form.  Let $\pi:G \to G\gamma_3(G)$ be the natural projection. Then by Remark \ref{r: from_s_to_2} and Lemma \ref{l: centralizers},  $\pi(a_n)$ has centralizer $\{\pi(a_n)^t z \mid t\in \mbb{Z}, \ z \in Z(G/\gamma_3(G))\}$. Furthermore,  Corollary \ref{c: infinite_order} implies that $[a_{n-1}, a_n]$ has infinite order in $G$, from which it follows that the centralizer of $\pi(a_n)$ projects onto an infinite cyclic group in  $(G/\gamma_3(G))/Z(G/\gamma_3(G))$. Hence, $\pi(a_n)$ is a centralizer-small element and by Theorem \ref{mainThmEq2}, the ring $\mbb{Z}$ is e-interpretable in $G/\gamma_3(G)$.  Thus $\mbb{Z}$ is e-interpretable in $G$ by Proposition \ref{any_step} and transitivity of e-interpretations. Consequently, $\mc{D}(G)$ is undecidable, by Corollary \ref{RedCor}.
\end{proof}

\begin{thm}\label{t: largest_ring}
Let $G$ be a finitely generated nilpotent group of class $2$ given by a finite full rank presentation $G = \langle a_1, \ldots,a_n \mid  r_1, \ldots,r_m \rangle_{\mc{N}_c}$.  Then, if $m \leq n-2$, the largest ring of scalars of $G$ is $\mbb{Z}$. 
\end{thm}
\begin{proof}
In the proof of the previous Theorem \ref{D} we showed that, under the hypothesis of the present theorem and after using Lemma \ref{SmithLemma2} to bring $M(A, R)$ into Smith normal form, the element $a_n$ is centralizer-small. The result is now a direct consequence of Theorem \ref{CSmallMaxZ}.
\end{proof}

\subsection{Direct decompositions}

\begin{lemma}
\label{le:dir-dec}
Let $G$ be a finitely generated nilpotent group which has a non-abelian free nilpotent subgroup of finite index. If $G = G_1 \times G_2$ then either $G_1$ or $G_2$ is finite.
\end{lemma}
\begin{proof}
Let $N$ be a non-abelian free nilpotent subgroup of $G$ of finite index $d$. Suppose that $G = G_1 \times G_2$. Then $G^d \leq N$, and $G^d = G_1^d \times G_2^d$. Observe also that $G' = G_1' \times G_2'$. 

Assume now that both $G_1$ and $G_2$ are infinite. Then by Lemma \ref{SegalCor2} the groups $G_i/G_i'$ are also infinite, $i = 1,2$. In particular, they both have elements of infinite order. Let $g_i \in G_i$ be elements such that  the image of $g_i$ in $G_i/G_i'$ is infinite, $i = 1,2$.  One has $g_i^d \in N$, and  $g_i^d \not \in G_i'$. Hence $g_i^d \not \in G'$, in particular, $g_i^d \not \in N'$, $i = 1,2$. The image of the subgroup $H = \langle g_1^d, g_2^d\rangle$ in $G/G'$ is isomorphic to $\mathbb{Z}^+ \times \mathbb{Z}^+$ (a direct product of two infinite cyclic groups). Since $N' \leq G'$ it follows that the image of $H$ in $N/N'$ is also isomorphic to $\mathbb{Z}^+ \times \mathbb{Z}^+$. Since $N$ is non-abelian free nilpotent one has that $Z(N) \leq N'$, so the image of $H$ in $N/Z(N)$ is also isomorphic to $\mathbb{Z}^+ \times \mathbb{Z}^+$. 
Now observe, that $g_2^d \in C_N(g_1^d)$, so $H \leq C_N(g_1^d)$ and the image of $H$ in  $C_N(g_1^d)/Z(N)$ is again  $\mathbb{Z}^+ \times \mathbb{Z}^+$. However, it is known (see, for example, Proposition 5.1 of \cite{LM}) that the image of the centralizer $C_N(g)$  of any element $g \in N \smallsetminus N'$ is cyclic in the quotient $N/Z(N)$ - contradiction with the claim above that  the image of $H$ in $N/Z(N)$ is  isomorphic to $\mathbb{Z}^+ \times \mathbb{Z}^+$.  This shows that one of the groups $G_1$ or $G_2$ is finite, as claimed.
\end{proof}

\begin{thm}\label{t: direct_decompositions}
Let $G$ be a finitely generated nilpotent group of class $c \geq 2$ given by a finite full rank presentation $G = \langle a_1, \ldots,a_n \mid  r_1, \ldots,r_m \rangle_{\mc{N}_c}$.  
If $m \leq n-1$ then  in any direct decomposition of $G$ all, but one, direct factors are finite. 
\end{thm}

\begin{proof}
The case  $m\leq n-2$ follows directly from the previous Lemma \ref{le:dir-dec} and the fact that $G$ contains a free nilpotent subgroup of finite index due to Theorem \ref{A}.

 If $m = n-1$, then Theorem \ref{A} ensures that $G$ has an infinite cyclic subgroup $\langle c \rangle$ of finite index, say $d$. Hence if we had $G = G_1 \times G_2$ then $g_1^d \in \langle c \rangle$ for any $g_1 \in G_1$. Writing $c=c_1c_2$ with $c_1 \in G_1, c_2 \in G_2$, we have  that either $c_2$ has finite order or $g_1^d = 1$ for all $g_1 \in G_1$, in which case $G_1$ is finite. A symmetric statement holds for $c_1$ and $G_2$. Finally, since either $c_1$ or $c_2$ has infinite order, either $G_2$ or $G_1$ is finite, respectively.
\end{proof}

Notice that if $m\geq n$ then $G$ is finite due to Theorem \ref{A}, and there is nothing to say regarding the  decomposability of $G$ into infinite direct factors.

\section{Random finite presentations and random walks}\label{s: random_section}

Recall that, given some words $R=\{r_1, \dots, r_m\}$   on  $A^{\pm 1}=\{a_1^{\pm 1}, \dots, a_n^{\pm 1}\}$, we let  $M(A, R)$ denote the $m \times n$ matrix whose $(i,j)$-th entry is the sum of the exponents of all the $a_j$'s appearing in $r_i$.

In this section we prove:

\begin{thm}\label{B}
Let $R$ be a set of $m$ words of length $\ell$, each one obtained by successively concatenating  randomly chosen letters from $A^{\pm 1} = \{a_1^{\pm 1}, \dots, a_n^{\pm 1} \}$ with uniform probability. Then $M(A, R)$ has full rank (i.e.\ ${\it rank}(M(A, R))=\min\{n, m\}$) asymptotically almost surely as $\ell \to \infty$.
\end{thm}
We next describe in detail how the matrix $M(A, R)$ is related to the words $R$. Let $r_i \in R$ be a word of length $\ell$ obtained as in the statement of Theorem \ref{B}. I.e., $r_i$ is obtained first by randomly choosing a letter $x_{i,1}\in A^{\pm 1}$, then by choosing another letter $x_{i,2}$ and concatenating $x_{i,1}$ with $x_{i,2}$, and so on until we obtain the ``random'' word $w= x_{i,1} \dots x_{i,\ell}$. Projecting $w$ onto the abelianization $F/F' \cong \mbb{Z}^n$ of the free group $F(A)= F$ we obtain the following equality: 
\begin{equation}\label{e: random_word}
r_i F' = a_1^{\ell_{i,1}}, \dots, a_n^{\ell_{i,n}} F'
\end{equation} 
for some integers $\ell_{i,j}$ and $c_i \in F'$. Then the $i$-th row of $M(A, R)$  is precisely  $(\ell_{i,1}, \dots, \ell_{i,n})$.

It is worth observing what is happening in the Cayley graph $\Gamma$ of $F/F'$ each time a letter $x_{i,j+1}$ is concatenated with the  word $r_{i,j} =_{def} x_{i, 1} \dots x_{i,j}$. The element $r_{i,j}F'$ corresponds to a vertex $v$ of $\Gamma$, and multiplication of $r_{i,j}F'$ by $x_{i,j+1}F'$ results in taking a step  along the edge outgoing from $v$ with label $x_{i,j+1}$.  Once the $\ell$ letters have been concatenated we will have arrived at the vertex of $\Gamma$ with label \eqref{e: random_word}. Thus each row of $M(A, R)$ is obtained by taking a \emph{random walk} in the Cayley graph of $\mbb{Z}^n$, starting at the origin (we remark that our random walks allow backtracking).  Thus Theorem \ref{B} can be rephrased as follows: 
\begin{thm}[Reformulation of Theorem \ref{B}]\label{B'}
The final positions of $m\leq n$ independent random walks of length $\ell$ in $\mathbb{Z}^{n}$ are  linearly independent asymptotically almost surely as $\ell \to \infty$, seen as vectors of $\mathbb{R}^n$.
\end{thm}
In the next subsection we provide a more formal treatment of the notion of random walk. A comprehensive treatment of this subject can be found in  \cite{Lawler}, for example. 

\subsection{Central Limit Theorems}
The goal of this subsection is to prove Proposition \ref{distribution_bound} (see below), which provides an upper bound on the probability that a random walk ends at a specific point of $\mbb{Z}^n$. 

Let $A= \{a_1, \dots, a_n\}$ be a basis of $\mbb{Z}^n$. We identify each $a_i$ with the $n$-vector that has $1$ in its $i$-th component and $0$ in all other components. Correspondingly, in what follows we use additive notation for $\mbb{Z}^n$.

One can model a \emph{random walk} in $\mbb{Z}^n$ (with respect to the basis $A$) in the following way: Let $X$ be a random variable taking values in $A^{\pm 1}$ with uniform probability. Then $X$ can be interpreted as a random variable that indicates the position one will be in after taking a random step in the Cayley graph of $\mbb{Z}^n$ (with respect to the basis $A$, and starting at the origin vertex, denoted $0$). For this reason one says that $X$ is a \emph{one-step random walk} in $\mbb{Z}^n$. This idea can be easily extended to random walks of arbitrary length $\ell$: in this case one takes a sequence of $\ell$ independent identically distributed (i.i.d.)\ random variables $X_1, \dots, X_{\ell}$, each one taking values in $A^{\pm 1}$ with uniform probability. Then the random variable $S_{\ell}=X_1+\ldots+X_{\ell}$ is called a \emph{$\ell$-step random walk in $\mathbb{Z}^n$}. Similarly as before, $S_{\ell}$ can be thought of as a random variable that indicates the vertex one will be in after taking $\ell$ random steps in the Cayley graph of $\mbb{Z}^n$, if one starts at the origin $0$. An (infinite) \emph{random walk} $S$ is, informally speaking, the limit of $S_{\ell}$ as $\ell$ tends to infinity.  Formally, $S$ is defined as the sequence  $(S_0,S_1,\dots,S_{\ell},\dots)$, with $S_0=0$. 
 
Let now $X_{\ell}$ be one of the random variables introduced above.  We can identify $X_{\ell}$ with a tuple of random variables $(x_{\ell,1},\dots,x_{\ell,n})$, in which case one sees that, for each $i=1, \dots, n$, the coordinate  $x_{\ell,i}$ takes values $1, -1, 0$ with probabilities $1/2n, 1/2n$ and $1-1/n$, respectively. Notice that the random variables $x_{\ell,1}, \dots, x_{\ell,n}$ are not independent (for example, if one of them takes the value $1$ then the others must take the value $0$). On the other hand, since the $X_{\ell}$'s are i.i.d., so are all  variables in the set $\mathcal{X}_i = \{x_{\ell,i} \mid \ell \geq 0\}$, for a fixed $i$. Similarly, one can write $S_\ell$ in coordinate form: $S_{\ell} = (s_{\ell, 1}, \dots, s_{\ell,n})$. Then $s_{\ell, i} = x_{1,i} + \dots + x_{\ell, i}$ for all $i=1, \dots, n$. We now  use the Central Limit Theorem (CLT) to find the asymptotic behavior (as $\ell \to \infty$, with $i$ fixed) of each random variable $s_{\ell,i}$.
Observe that each $x_{\ell,i}$ has expected value and variance  
\begin{align}
&\mbb{E}\left( x_{\ell,i} \right) = 1\frac{1}{2n}+(-1)\frac{1}{2n} = 0, \\  
&{\it Var}\left(x_{\ell,i}\right)= \mbb{E}\left(\left(x_{\ell,i} - \mbb{E}(x_{\ell,i})\right)^2\right) = \mbb{E}\left(x_{\ell,i}^2\right) - \mbb{E}\left(x_{\ell,i}\right)^2=\frac{2}{2n}=\frac{1}{n},
\end{align}
and hence, by the CLT, $s_{\ell,i}/{\sqrt{\ell}}$ converges in distribution to the normal distribution $N(0, 1/n)$  with expectation $0$ and standard deviation $1/\sqrt{n}$. More precisely, for every $i=1,\dots, n$, and $M,N\in \mathbb{R}\cup\{\pm\infty\}$, 
\begin{equation}\label{conv_dist}
\lim_{\ell \to \infty} \left[ \mathbb{P}\left( N<\frac{s_{\ell,i}}{\sqrt{\ell}}< M \right)  \right]= \mathbb{P}\left( N < \xi < M \right),
\end{equation} 
where $\xi$ is a random variable with distribution $N(0, 1/n)$. Moreover, the $x_{\ell,i}$'s  have  finite third moment, indeed $\mbb{E}(|x_{\ell,i}|^3)=1 < \infty$. Thus, by the Berry-Esseen Theorem, the convergence in \eqref{conv_dist} occurs uniformly in the following way: there exists a constant $C$ such that, for all $i, N, M$ and $\ell$,
\begin{equation}\label{conv_dist_2}
\biggl\lvert  \mathbb{P}\left(N<\frac{s_{\ell,i}}{\sqrt{\ell}}<M\right)- \mathbb{P}\left(N<\xi<M\right)  \biggl\lvert \leq \frac{C}{\sqrt{\ell}}.
\end{equation} 
We are now in a position to prove the following:
\begin{lemma}\label{escape_speed}
Given a sequence $\epsilon_\ell$ with $ \lim_{\ell\to \infty} \epsilon_\ell = \infty$, the following holds for all $i$:
$$
\lim_{\ell\to \infty} \left[\mathbb{P}\left(\biggl|\frac{s_{\ell,i}}{\sqrt{\ell}}\biggr|\geq\epsilon_\ell\right)\right]= 0. 
$$ 
\end{lemma}
\begin{proof}
Indeed, it suffices to show that $\mathbb{P}(|s_{\ell,i}/\sqrt{\ell}| < \epsilon_\ell) \to 1$. First notice that $\mathbb{P}(|\zeta|<\epsilon_\ell)\to 1$ for any random variable $\zeta$ following a uniform distribution. Hence, by \eqref{conv_dist_2},
$$
\mathbb{P}\left(|(s_{\ell,i}/\sqrt{\ell})| < \epsilon_\ell\right) \geq   \left(  \mathbb{P}\left(|\xi|<\epsilon_\ell\right) - C/\sqrt{\ell} \right) \to 1.
$$
\end{proof}
Intuitively, this result shows that, for large $\ell$, the variables $s_{\ell,i}$ have absolute value not larger than $\e_\ell\sqrt{\ell}$ almost surely. For our arguments we will use $\epsilon_\ell=\ln(\ell)$, though it suffices to take any sequence that approaches infinity slowly enough.

We will also need the following  local version of the Central Limit Theorem for random walks. Given a point $T=(t_1, \dots, t_n)\in \mathbb{Z}^n$, denote by $p_\ell(T)$ the probabilty of being at  $T$ on the $\ell$-th step of a random walk, i.e.\ $p_\ell(T)=\mathbb{P}(S_\ell=T)$.
\begin{LCLT}\label{Local_CLT}
Following the notation above, there exists a constant $c_0$ such that, for all $T\in \mathbb{Z}^n$ and $\ell\in \mathbb{N}$,
$$
\biggl|p_\ell(T)+p_{\ell+1}(T)-2\tilde{p}_\ell(T)\biggr|< \frac{c_0}{\ell^{(n+2)/2}},
$$ 
where 
$$
\tilde{p}_\ell(T)= \frac{1}{\left(2\pi \ell\right)^{n/2}c_1} e^{-\frac{J(T)^2}{2\ell}},
$$
and $c_1$ is a positive constant and $J(\cdot)^2$ is a positive definite quadratic form.
\end{LCLT}
This provides an  upper bound for $p_\ell(T)$:
\begin{prop} \label{distribution_bound}
There exists a  constant $c_3$ such that, for all $T\in \mathbb{Z}^n$ and $\ell\in \mathbb{N}$,
$$
p_\ell(T)\leq \frac{c_3}{\ell^{n/2}}. 
$$ 
\end{prop}
\begin{proof}
By the Local Central Limit Theorem there exist positive constants $c_2, c_3$ such that
$$
p_\ell(T) \leq |p_\ell(T)+p_{\ell+1}(T)|\leq 2\tilde{p}_\ell(T)+  \frac{c_0}{\ell^{(n+2)/2}}\leq \frac{c_2}{\ell^{n/2}}+\frac{c_0}{\ell^{(n+2)/2}}\leq \frac{c_3}{\ell^{n/2}}.$$ 
\end{proof}

\subsection{Proof of linear independence}

Consider $m$ independent random walks of $\ell$ steps in $\mathbb{Z}^n$, $$S_{j,\ell}= (s_{j,\ell, 1}, \dots, s_{j, \ell, n}), \quad j=1, \dots, m,$$ where here and in what follows we  maintain the notation of the previous section,  adding an extra subindex $j$ (ranging from $1$ to $m$) when appropriate. Let $M_\ell$ be  the $m \times n$ matrix whose $j$-th row consists in the components of $S_{j,\ell}$, i.e.\  $(s_{j,\ell,1}, \dots, s_{j,\ell,n})$, ($j=1, \dots, m$).

The goal of this subsection is to prove Theorem \ref{B}. This result states that
\begin{equation}\label{random_matrix_is_max_rank}
\lim_{\ell\to \infty}\mathbb{P}\left({\rm rank}(M_\ell)={\rm min}(m,n)\right)= 1.
\end{equation}
In other words, that \emph{the matrices $M_\ell$ have full rank asymptotically almost surely as $\ell$ tends to infinity.} As explained at the beginning of Section \ref{s: random_section}, this is equivalent to proving that the $m$ independent random walks $S_{1,\ell}, \dots, S_{m,\ell}$ finish at $m$ linearly independent vectors of $\mbb{Z}^n$ asymptotically almost surely as $\ell$ tends to infinity (see Theorem \ref{B'}).
Given an $m\times n$ matrix $M=(m_{i,j}\mid i,j)$ with integer  entries, let $f(M)$ be the polynomial 
$$
f(M)=f\left(m_{i,j} \mid i,j\right)=\sum  \left[{\rm det}(M_0)\right]^2,
$$ 
where the sum runs over all maximal minors $M_0$ of $M$. Of course, $f(M)=0$ if and only if  ${\rm det}(M_0)=0$ for all $M_0$, i.e.\ if and only if $M$ does not have full rank. We will need the following  combinatorial result to estimate the number of roots of $f$ in any given finite set:
\begin{SZL}\label{Schwartz}
Let $f(\mathbf{x})=f(x_1,\ldots,x_N)\in \mathbb{C}[x_1,\ldots,x_N]$  be a polynomial of degree $d$ on $N$ variables,  and let $I$ be a finite set of complex numbers. Then
$$
\bigl|\left\{ \mathbf{x} \in I^N \mid f(\mathbf{x})=0 \right\}\bigr| \leq d |I|^{N-1}.
$$
\end{SZL}
\begin{proof}[\textbf{\emph{Proof of Theorem \ref{B}}}]
We need to show that $\mathbb{P}\left({\rm rank}(M_\ell)={\rm min}(m,n)\right) \to  1$ as $\ell\to \infty$, or, equivalently, that $\mathbb{P}\left(f(M_\ell)=0\right)\to 0$.  Write $\epsilon_\ell = \ln(\ell)$, and let $\mathcal{M}_\ell$ be the set of $m \times n$ integer matrices $M$ such that $|m_{i,j}|  < \sqrt{\ell} \epsilon_\ell$ for all entries $m_{i,j}$ of $M$. Then

\begin{equation}\label{eq2}
\mathbb{P}\left( f(M_\ell)=0 \right) \leq  \mathbb{P}\left( f(M_\ell)=0, \ M_\ell \in \mathcal{M}_\ell \right) +  \mathbb{P}\left( M_\ell \notin \mathcal{M}_\ell \right)
\end{equation}
for all $\ell$. Recall that the $j$-th row of $M_\ell$ is $S_{j,\ell}=\left(s_{j,\ell,i} \mid i=1,\dots, n \right)$. By Lemma \ref{escape_speed},
\begin{align*}
\mathbb{P}\left( M_\ell \notin \mathcal{M}_\ell \right) &\leq \sum_{j=1}^m \mathbb{P} \left( S_{j, \ell} \ \hbox{is such that} \  
|s_{j,\ell,i}| \geq \sqrt{\ell} \epsilon_\ell \ \hbox{for some} \ i \right)\\
&\leq \sum_{j,i}  \mathbb{P}\left( S_{j, \ell} \ \hbox{is such that} \ 
|s_{j,\ell,i}| \geq \sqrt{\ell} \epsilon_\ell \right) \to 0.
\end{align*} 
We now  bound the first summand of \eqref{eq2}.  First notice that 
$$
\mathbb{P}\left( f(M_\ell)=0, \  M_\ell \in \mathcal{M}_\ell \right) = \sum_{i=1}^{k_\ell} \mathbb{P}\left( M_\ell = T_{\ell,i} \right),
$$
where $T_{\ell,1}, \dots, T_{\ell, k_\ell}$ are all the zeros of $f$ in $\mathcal{M}_\ell$ (identifying  $m \times n$ matrices with vectors from $\mathbb{Z}^{mn} \subseteq \mbb{C}^{mn}$).
By the Schwartz-Zippel Lemma, 
$$
k_\ell\leq d\left(2\sqrt{\ell}\epsilon_\ell +1\right)^{m n-1} \leq d\left(3 \sqrt{\ell}\epsilon_\ell\right)^{mn-1},
$$
where  $d$  is the degree of $f$.
Now observe that, for fixed $\ell$ and $i$, $M_\ell=T_{\ell,i}$ if and only if the $j$-th row $S_{j,\ell}$ of $M_\ell$ is equal to the $j$-th row of $T_{\ell,i}$, which we denote $T_{j,\ell,i}$, for all $j=1, \dots, m$. Using Proposition \ref{distribution_bound} and the fact that the $S_{j,\ell}$  ($j=1,\dots,m$) form a set of $m$ independent random walks, 
$$
\mathbb{P}\left( M_\ell=T_{\ell,i}\right) = \prod_{j=1}^{m} \mathbb{P}(S_{j,\ell} = T_{j,\ell,i}) \leq \left( \frac{c_3}{\ell^{n/2}} \right)^m
$$
for every $i=1, \dots, k_\ell$. So far $\varepsilon_{\ell}$  was an arbitrary sequence tending to infinity. Taking $\varepsilon_{\ell} = \ln(\ell)$ we obtain
\begin{align}
\mathbb{P}\left( f(M_\ell)=0, \ M_\ell \in \mathcal{M}_\ell \right) \leq  d\left(3 \sqrt{\ell}\epsilon_\ell\right)^{m n-1}\left(  \frac{c_3}{\ell^{mn/2}} \right)^m \leq c_4 \frac{1}{\sqrt{\ell}},\nonumber
\end{align}
for some positive constant $c_4$. 
Hence $\mathbb{P}\left(f(M_\ell)=0\right) \to 0$, as needed.
\end{proof}

\section{Random nilpotent groups}

\label{sec:random}

In this section we study random nilpotent groups according to the few-relators model.
More precisely, we consider a group presentation 
\begin{equation}\label{e: main_presentation2}
G=\langle a_1, \dots, a_n \mid r_1, \dots, r_m\rangle_{\mathcal{N}_c} = \langle A \mid R \rangle_{\mc{N}_c}
\end{equation}
in the variety of $c$-step nilpotent groups, where $R$ is a set of $m$ words of length $\ell$ chosen randomly as explained in Section \ref{s: random_section}. We then study the asymptotic properties of $G$ as $\ell$ tends to infinity. 

The key observation is that due to Theorem \ref{B} we know that, asymptotically almost surely as $\ell \to \infty$, the presentation \eqref{e: main_presentation2} is of full rank. Therefore, all results obtained so far for nilpotent groups given by a full rank presentation hold also for `random' nilpotent groups asymptotically almost surely. We next present the corresonding results. From Theorem \ref{A} we get:

\begin{thm}
Let $n, m, c \in \mbb{N}$, and let $G$ be a finitely generated $c$-step nilpotent group given by a presentation $\langle a_1, \dots, a_n \mid r_1, \dots, r_m \rangle_{\mc{N}_c}$, where all relators $r_i$ have length $\ell$. Then  the following holds asymptotically almost surely  as $\ell \to \infty$: There exists a subset $A_0\subseteq \{a_1, \dots, a_n\}$ with $|A_0| = n-m$ ($A_0=\emptyset$ if $m\geq n$) such that the following is true:
\begin{enumerate}
\item If $m\geq n$, then $G$ is finite.
\item If $m=n-1$, then $\langle A_0\rangle$ is infinite cyclic and has finite index in $G$.%
\item If $m \leq n-2$, then $\langle A_0 \rangle$ is a free $c$-step nilpotent subgroup of rank $n-m$ which has finite index in $G$.
\end{enumerate}
 Furthermore, $A_0$ can be chosen to be  any set of $n-m$ generators $\{a_{i_j}\mid 1\leq j\leq n-m\}$ such that the rank of the matrix $M(A, R)$ coincides with the rank of the matrix obtained from $M(A, R)$ after removing its $i_1, \dots, i_{n-m}$-th columns. 
\end{thm}

One can say more if the nilpotency class is $2$. From Lemma \ref{l: centralizers} and Theorem \ref{t: center} we have:
\begin{prop}\label{t: random_2_step}
Let $ n, m \in \mbb{N}$, and let $G$ be a finitely generated $2$-step nilpotent group given by a presentation $\langle a_1, \dots, a_n \mid r_1, \dots, r_m \rangle_{\mc{N}_2}$, where all relators $r_i$ have length $\ell$. Then if $m \leq n-2$ the following holds asymptotically almost surely  as $\ell \to \infty$: $G$ is isomorphic to a group given by a full-rank presentation $\langle b_1,\ldots,b_n\mid r_1',\ldots,r_m'\rangle_{\mc{N}_2}$ in Smith normal form, which implies $Z(G) = G'$ and $C_G(b_i) = \langle b_i\rangle Z(G)$ for all $i = m+1, \dots, n$.
\end{prop}

Recall that  Theorem \ref{t: Torsion_in_Random_step2} states that $T(G)$ is trivial if and only if $T(G/G')$ is trivial (and also that the latter occurs if and only if $G$ is free nilpotent); where here $G$ is a $2$-step nilpotent group given by a full rank presentation, and $T(G)$ denotes the torsion subgroup of $G$. Consequently, the asymptotic probability that a f.p.\ $2$-step nilpotent group $G$ has trivial torsion is precisely the asymptotic probability that $G$ is free nilpotent, which in turn coincides with \emph{the asymptotic probability that a f.p.\ abelian group has trivial torsion}. To the best of our knowledge, the latter is not known and we leave it as an open question. As a partial contribution we prove that this probability is not one. For this we  use some of the results in \cite{Duchin} ---in particular the statement regarding torsion when $m=1$ is proved in such reference.

\begin{prop}
Let $ n, m \in \mbb{N}$, and let $G$ be a finitely generated $2$-step nilpotent group given by a presentation $\langle a_1, \dots, a_n \mid r_1, \dots, r_m \rangle_{\mc{N}_2}$, where $m\leq n-1$ and  all relators $r_i$ have length $\ell$. Let $p_t$ be  the asymptotic probability as $\ell \to \infty$ that $G$ has trivial torsion subgroup, and let $p_f$ be the asymptotic probability that $G$ is free nilpotent. Then $p_t = p_f < 1$. If $m= 1$ then these probabilities are precisely $1 - \frac{1}{\zeta(n)}$, where $\zeta$ is Riemann's zeta function.
\end{prop}
\begin{proof}
Let $M(A, R)$ be the relation matrix of the presentation  $G=\langle A \mid R \rangle_{\mc{N}_2}=\langle a_1, \dots, a_n \mid r_1, \dots, r_m \rangle_{\mc{N}_2}$. Due to Theorem \ref{B}, we can assume without loss of generality that this presentation is of full-rank and thus that $M(A, R)$ also has maximum rank $m$. Let $d$ be the greatest common divisor of the determinants of all $m \times m$ minors of $M(A, R)$. It follows from Proposition 4.3, Chapter 8 of \cite{sims} that if $d \neq 1$ then the Smith normal form of $M(A, R)$ has an entry which is neither $1$ nor $-1$. In such a case, by Lemma \ref{SmithLemma2}  the abelianization of $G$ has non-trivial torsion subgroup, and Theorem \ref{t: Torsion_in_Random_step2} then implies that $G$ has non-trivial torsion subgroup.  Hence it suffices to prove that $d\neq 1,-1$ with non-zero probability. For this, let $d'$ be the greatest common divisor of the entries in the first row of $M(A, R)$, and let $M'$ be an arbitrary $m\times m$ submatrix of $M(A,R)$. Let $1\leq i\leq m$ be such that $M'$  starts with the $i$-th column of $M(A, R)$ and continues up to the $(i+m-1)$-th column of $M(A, R)$. Denote by $r_{s,t}$ the entries of $M(A, R)$ ($1\leq s \leq m$, $1\leq t \leq n$). Then $$det(M')= \sum_{j=1}^{m} (-1)^{j-1} r_{1, j+i} M_j'$$ for certain $m-1 \times m-1$ submatrices $M_j'$ of $M'$.  Therefore $d'$ divides $det(M')$. Since $M'$ was an arbitrary $m\times m$ minor of $M(A,R)$, $d'$  divides  the determinants of all $m\times m$ minors of $M(A,R)$, and so $d'$ divides $d$. Since the asymptotic probability that $d'\neq 1, -1$ is precisely $1- 1/\zeta(n)$ (see Corollary 17 of \cite{Duchin}), we have that $d \neq 1, -1$ with asymptotic probability at least  $1- \frac{1}{\zeta(n)}$. If $m= 1$ then $d=d'$ and then the  asymptotic probability that $G$ has non-trivial torsion subgroup is precisely $1- \frac{1}{\zeta(n)}$.
\end{proof}

Regarding regularity and the properties of being QFA and first-order rigid using Theorem \ref{th:FAQ}:
\begin{thm}
Let $n, m, c \in \mbb{N}$, and let $G$ be a finitely generated $c$-step nilpotent group given by a presentation $\langle a_1, \dots, a_n \mid r_1, \dots, r_m \rangle_{\mc{N}_c}$,  where $c \geq 2$ and all relators $r_i$ have length $\ell$. Then the following holds asymptotically almost surely  as $\ell \to \infty$: 
If $m \leq n-2$  then  $G$ is QFA, in particular it is first-order rigid.
\end{thm}

Next we address the Diophantine problem of a random nilpotent group using Theorem \ref{D}:

\begin{thm}
Let $n, m, c \in \mbb{N}$, and let $G$ be a finitely generated $c$-step nilpotent group given by a presentation $\langle a_1, \dots, a_n \mid r_1, \dots, r_m \rangle_{\mc{N}_c}$, where $c\geq 2$, and all relators $r_i$ have length $\ell$. Then  the following holds asymptotically almost surely  as $\ell \to \infty$: 
\begin{enumerate}
\item If $m \leq n-2$, then the ring $\mathbb{Z}$ is e-interpretable in $G$  and the Diophantine problem in $G$ is undecidable.
\item If $m \geq n-1$  then the Diophantine problem in $G$ is decidable. 
\end{enumerate}
\end{thm}

Finally, the corresponding result from the Theorem \ref{t: direct_decompositions} for direct decompositions of random nilpotent groups is:

\begin{thm}
Let $n, m, c \in \mbb{N}$, and let $G$ be a finitely generated $c$-step nilpotent group given by a presentation $\langle a_1, \dots, a_n \mid r_1, \dots, r_m \rangle_{\mc{N}_c}$, where $c\geq 2$, and all relators $r_i$ have length $\ell$. If $m \leq n-2$ then the following holds asymptotically almost surely as $\ell \to \infty$:  In any direct decomposition of $G$ all, but one, direct factors are finite. 
\end{thm}

\section{Acknowledgements}

The authors are grateful to an anonymous referee whose comments helped improve and shorten the proof of Theorem \ref{A}.

This work was supported by  Russian Science Foundation grant project  19-11-00209.

Additionally, the first named author was supported by the ERC grant PCG-336983, by the Spanish Government grant IT974-16, and by the Ministry of Economy, Industry and Competitiveness of the Spanish Government Grant MTM2017-86802-P.

\bibliography{bib}

\end{document}